\documentclass[reqno, 12pt]{amsart}
\usepackage[T1]{fontenc}    
\usepackage[utf8]{inputenc} 
\usepackage{lmodern}       
\usepackage{mathrsfs}
\usepackage{amscd}
\usepackage{amsmath}
\usepackage{latexsym}
\usepackage{amsfonts}
\usepackage{amssymb}
\usepackage{amsthm}
\usepackage{graphicx}
\usepackage{makecell}
\usepackage{array}
\usepackage{booktabs}
\usepackage{multirow}
\usepackage{color,xcolor}
\usepackage{comment}
\allowdisplaybreaks[4] 

\usepackage{fancyhdr}
\newtheorem{theorem}{Theorem}[section]
\newtheorem{proposition}[theorem]{Proposition}

\newtheorem{lemma}[theorem]{Lemma}
\newtheorem{definition}[theorem]{Definition}

\newtheorem{remark}[theorem]{Remark}

\usepackage[colorlinks=true, linkcolor=blue]{hyperref} 
\numberwithin{equation}{section}

\title{inverse boundary value problems of determining nonlinear coefficients for the JMGT equation } 
\author[Qiu]{Dong Qiu}
\address{School of Mathematical Sciences, Zhejiang University}
\email{qiudong@zju.edu.cn}

\author[Xu]{Xiang Xu}
\address{School of Mathematical Sciences, and Center for Interdisciplinary Applied Mathematics, Zhejiang University}
\email{xxu@zju.edu.cn}

\author[Ye]{Yeqiong Ye}
\address{School of Mathematical Sciences, Zhejiang University}
\email{yeyeqiong@zju.edu.cn}

\author[Zhou]{Ting Zhou}
\address{School of Mathematical Sciences, Zhejiang University}
\email{ting\_zhou@zju.edu.cn}
\subjclass[2020]{35R30, 35L05, 44A12}
\keywords{Jordan--Moore--Gibson--Thompson equation, inverse boundary value problem, second order linearization, Gaussian beam solutions, Dirichlet-to-Neumann map}

\begin{document}
\begin{abstract} 
We consider inverse boundary value problems for the Jordan-Moore-Gibson-Thompson (JMGT) equation in nonlinear acoustics with quadratic nonlinearities of Kuznetsov-type and Westervelt-type. 
We show that the associated boundary Dirichlet-to-Neumann map uniquely determines the nonlinear coefficients $\beta$ in the Westervelt-type model, and the pair $(\beta,\kappa)$ in the Kuznetsov-type model, provided that the observation time is greater than the maximal boundary-to-boundary geodesic travel time. 
The results are obtained in both the Euclidean setting and on compact Riemannian manifolds with proper geometric assumptions. 
The proof is based on the idea of second order linearization combined with the construction of geometric optics and Gaussian beam solutions, reducing the inverse problem of uniqueness to the injectivity of associated geodesic ray transforms.
\end{abstract}

\maketitle
\section{Introduction}
\subsection{Statement of the Problem}
In this article, we study inverse boundary value problems for the Jordan--Moore--Gibson--Thompson (JMGT) equation and ask whether its nonlinear coefficients can be uniquely determined from boundary measurements.

JMGT equations arise in nonlinear acoustics as higher-order models for high-intensity sound propagation, extending the classical Westervelt and Kuznetsov equations by incorporating thermoviscous dissipation and relaxation effects. As third-order-in-time models, they capture finite-speed wave propagation together with frequency-dependent attenuation, and are therefore particularly relevant in applications such as high-intensity medical ultrasound. We study the JMGT inverse problem on a Riemannian manifold in order to model inhomogeneous media and to relate high-frequency propagation to geodesic ray transforms.

Let $(M,g)$ be a smooth compact Riemannian manifold of dimension $n\geq2$ with smooth boundary $\partial M$.
We also assume that $T>0$, $\mathcal{M}:=[0,T] \times M$, $\alpha,b,c \in C^{\infty}(M)$ and $b>0$,  and $\Gamma:=[0,T] \times \partial M$.
We denote by $\Delta_g$ the Laplace--Beltrami operator on $(M,g)$. In local coordinates,
\[
\Delta_g u=\sum_{i,j=1}^n|g|^{-\frac12}\partial_{x_i}\Bigl(|g|^{\frac12}g^{ij}\partial_{x_j}u\Bigr),
\]
where $g^{-1}=(g^{ij})_{1\le i,j\le n}$ and $|g|=\det(g_{ij})$.
We consider the JMGT equation with Kuznetsov-type nonlinearity in $\mathcal{M}$ 
\begin{equation}\label{JMGT}
\left\{
\begin{aligned}
&\partial_t^3u + \alpha \partial_{t}^2 u  - b\Delta_g \partial_t u-c^2\Delta_g u = \partial_t \left(\beta(x)(\partial_tu)^2+\kappa(x)\left| \nabla_g u  \right|_{g}^2   \right) 
    && \text{in } \mathcal{M}, \\
&u= f(t,x)
    && \text{on } \Gamma, \\
&u(0,x)=v_0(x), \partial_t u(0,x) = v_1(x), \partial_t^2u(0,x)=v_2(x) 
    && \text{on } M,
\end{aligned}
\right.
\end{equation}
where $\alpha$, $b$, and $c$ denote the frictional damping coefficient, the sound diffusivity, and the speed of sound, respectively. Here, $\beta$ denotes the acoustic nonlinearity coefficient associated with the nonlinear pressure–density relation of the medium, while 
$\kappa$ denotes the coefficient of the quadratic gradient nonlinearity, describing nonlinear effects induced by spatial gradients of the acoustic field. Both are key target coefficients in nonlinear ultrasonic tomography, and their accurate identification may provide valuable information for ultrasound-based diagnosis of diseases such as breast tumors and liver fibrosis.
We also consider the Westervelt-type nonlinearity, which may be viewed as a simplified model, given by
\begin{equation}\label{JMGT2}
\left\{
\begin{aligned}
&\partial_t^3u + \alpha \partial_{t}^2 u  - b\Delta_g \partial_t u-c^2\Delta_g u = \partial_t^2 \left(\beta(x)u^2   \right) 
    && \text{in } \mathcal{M}, \\
&u= f(t,x)
    && \text{on } \Gamma, \\
&u(0,x)=v_0(x), \partial_t u(0,x) = v_1(x), \partial_t^2u(0,x)=v_2(x) 
    && \text{on } M.
\end{aligned}
\right.
\end{equation}
For a pair of coefficients $(\beta,\kappa)$ and fixed zero initial data $(v_0,v_1,v_2)=(0,0,0)$, the corresponding \textit{Dirichlet-to-Neumann} (DtN) map for \eqref{JMGT} is  defined by
\begin{align}
&\Lambda^{\textrm{K}}_{\beta,\kappa}:f \mapsto c^2\partial_{\nu}u+b\partial_{\nu}\partial_t u \big|_{\Gamma},
\end{align}
where $u$ is the solution to \eqref{JMGT}, $\nu$ denotes the unit outward normal vector of $\partial M$ with respect to the Riemannian metric $g$, and $\partial_{\nu} u=g(\nabla_g u,\nu)$. For the Westervelt-type nonlinearity and fixed zero initial data $(v_0,v_1,v_2)=(0,0,0)$, we define its DtN map as
\[
\Lambda^\textrm{W}_{\beta}:f \mapsto c^2\partial_{\nu}u+b\partial_{\nu}\partial_t u \big|_{\Gamma},
\]
where $u$ is the solution to \eqref{JMGT2}.
We refer to Section \ref{wellposedness} for local well-posedness of the equations \eqref{JMGT} and \eqref{JMGT2}, which ensures that $\Lambda^{\textrm{K}}_{\beta,\kappa}$ and $\Lambda_{\beta}^{\textrm{W}}$ are well defined.
The inverse problems we consider are the following:
\begin{itemize}
\item Can we recover  $\beta$  from  the DtN map $\Lambda^\textrm{W}_{\beta}$?
\item Can we recover  $\beta$ and $\kappa$  from the DtN map  $\Lambda^\textrm{K}_{\beta,\kappa}$?
\end{itemize}

Before presenting our main results, we introduce the following notation and definitions. 
For a nonnegative integer $m$, we define the following function spaces:
\begin{align*}
H^{m}(\Gamma) &= \bigcap_{k=0}^{m} H^{k}\bigl([0,T]; H^{m-k}(\partial M)\bigr), \\
A_{m} &= \bigcap_{k=0}^{m+1} C^{k}\bigl([0, T]; H^{m+\frac{3}{2}-k}(\partial M)\bigr) \cap H^{m+2}(\Gamma), \\
N_{m+1} &=
\bigcap_{k=1}^{m+1}H^{k}\bigl([0, T]; H^{m+1-k}(\partial M)\bigr).
\end{align*}
We work in the \textit{energy spaces} $E^m$, defined by 
\begin{equation*}
E^m=\bigcap\limits_{k=0}^{m} C^k\left([0,T];H^{m-k}(M)\right),
\end{equation*}
equipped with the norm 
\begin{equation*}
\|u\|_{E^m}:=\sup_{t\in[0,T]}\sum_{k=0}^{m}\|\partial_t^k u(\cdot,t)\|_{H^{m-k}(M)}.
\end{equation*}
The space $E^m$ is an algebra if $m >n+1$ (see \cite{choquet-bruhat_general_2009}). Moreover, it satisfies the norm estimate
\begin{equation*}
||uv||_{E^m}\leq C_m ||u||_{E^m}||v||_{E^m},  \quad \text{for all } u,v\in E^m,
\end{equation*}
where $n$ is the dimension of $M$ and $C_m$ is a constant depending on $m$. 

\subsection{Main Results} 
For the following main results, we assume that $m$ is a positive integer and $m>n+1$. We introduce the compatibility conditions up to order $m+1$ required for the well-posedness results
\begin{equation}
\left\{
\begin{aligned}
&\partial_t^{k} f(0,x)=v_k(x),
\qquad &&\hspace{-4em}\text{on }\partial M,\ k=0,1,2,\\
&\partial_t^{k} f(0,x)
= -\alpha\,\partial_t^{k-1}u(0,x)
+ b\,\Delta_g \partial_t^{k-2}u(0,x)
+ c^{2}\,\Delta_g \partial_t^{k-3}u(0,x)\\  \label{compa1}
&\qquad\quad
\qquad &&\hspace{-4em}\text{on }\partial M,\ 3\le k\le m+1.
\end{aligned}
\right.
\end{equation}

We first state our well-posedness result for equation \eqref{JMGT}.
\begin{proposition} \label{propo1}
Fix $\beta \in C^m(\overline M)$ and $\kappa \in C^m( \overline M)$. We assume that $\beta=0, \ \kappa=0$ on $\partial M$. There exists a constant $\delta$ such that if the boundary and initial values $(v_0,v_1,v_2,f)$ belong to the set 
\begin{equation}\label{defN}
\begin{aligned}
\mathcal{N}_{\delta}
=
\Bigl\{
(v_0,v_1,v_2,f)
\in
H^{m+2}(M)\times H^{m+1}(M)\times H^{m}(M)\times A_m
:
\Bigr. \\
\|v_0\|_{H^{m+2}(M)}
+\|v_1\|_{H^{m+1}(M)}
+\|v_2\|_{H^{m}(M)}
+\|f\|_{H^{m+2}(\Gamma)}
< \delta,
\\
(v_0,v_1,v_2,f)
\text{ satisfy the compatibility conditions up to order } m+1
\text{ defined in } \eqref{compa1} 
\Bigl\}.
\end{aligned}
\end{equation}
then the nonlinear system \eqref{JMGT}  admits a unique solution $u \in E^{m+2}$ with $\partial_{\nu}u \in N_{m+1}$, and there exists a positive constant $C$ such that 
\begin{align} \label{estimate3}
\left\|u \right\|_{E^{m+2}}+\left\| \partial_{\nu}u  \right\|_{N_{m+1}}\leq  C   
\left( \left\|v_0\right\|_{H^{m+2}(M)}+ \left\|v_1\right\|_{H^{m+1}(M)}+\left\|v_2\right\|_{H^{m}(M)}+\left\|f\right\|_{H^{m+2}(\Gamma)} \right).
\end{align}
\end{proposition}
We next state our well-posedness result for equation \eqref{JMGT2}.
\begin{proposition} \label{propo2}
Fix $\beta \in C^m(\overline M)$. We assume that $\beta=0$ on $\partial M$. There exists a constant $\delta$ such that if the boundary and initial values $(v_0,v_1,v_2,f)$ belong to the set $\mathcal{N}_{\delta}$ defined in \eqref{defN}.
Then the nonlinear system \eqref{JMGT2}  admits a unique solution $u \in E^{m+2}$ with $\partial_{\nu}u \in N_{m+1}$, and there exists a positive constant $C$ such that 
\begin{align} \label{estimate5}
\left\|u \right\|_{E^{m+2}}+\left\| \partial_{\nu}u  \right\|_{N_{m+1}}\leq  C   
\left( \left\|v_0\right\|_{H^{m+2}(M)}+ \left\|v_1\right\|_{H^{m+1}(M)}+\left\|v_2\right\|_{H^{m}(M)}+\left\|f\right\|_{H^{m+2}(\Gamma)} \right).
\end{align}
\end{proposition}
We first state the main inverse results in the Euclidean setting. 
Let $M=\Omega \subset \mathbb{R}^n $ with $n \geq 2$ be a bounded and connected domain with a smooth boundary. We still denote $\Gamma:=[0,T] \times \partial \Omega$.
We introduce a parameter 
\begin{equation}
\gamma=\alpha-\frac{c^2}{b},
\end{equation}
which is related to exponential stability of the linearized system. We assume that  $\gamma>0$ in the Euclidean domain $\Omega$.
\begin{theorem}\label{thm1}
Assume that $T>\frac{1}{\sqrt{b}} \operatorname{diam}(\Omega)$ and fixed zero initial values $(v_0,v_1,v_2)=(0,0,0)$. We also assume that $\alpha,b,c$ in \eqref{JMGT} are constants. Let $\Lambda^\textrm{K}_{\beta_1,\kappa_1}$ and $\Lambda^\textrm{K}_{\beta_2,\kappa_2}$ denote the corresponding DtN maps for equation \eqref{JMGT} with coefficients $(\beta_1,\kappa_1),(\beta_2,\kappa_2) \in C^{m}(\overline\Omega) \times C^{m}(\overline\Omega)$ and assume that 
\begin{equation}
\beta_j=\kappa_j=0 \quad \text{on } \partial \Omega, \quad j=1,2.
\end{equation}
If there exists $\delta>0$ such that
\begin{equation*}
\Lambda^\textrm{K}_{\beta_1,\kappa_1}(f)=\Lambda^\textrm{K}_{\beta_2,\kappa_2}(f) \qquad 
\textrm{for all } f \in A_{m} \ \textrm{with } (0,0,0,f) \in \mathcal{N}_{\delta},
\end{equation*}
then one has
\begin{equation}
\beta_1=\beta_2,\ \kappa_1=\kappa_2   \quad \text{in } \Omega.
\end{equation}
\end{theorem}

For the non-Euclidean cases, we impose the following geometric assumptions. We first assume $\dim M=3$, $\alpha,b,c \in C^{\infty}(M)$ and 
\[  \operatorname{diam}_{b^{-1}g}(M):= \sup  \left\{ \text{lengths of all geodesics in } (M,b^{-1}g)\right\}<\infty  .  \]
We further impose the following geometric assumption. 

\textbf{Geometric Assumption}: Assume that $\partial M$ is strictly convex with respect to $b^{-1}g$ and  $(M, b^{-1}g)$ satisfies the foliation condition, that is, there exists a smooth strictly convex function $f:M \to \mathbb{R}$, as introduced in \cite{uhlmann_inverse_2016}.

For Westervelt-type nonlinearity, we have
\begin{theorem} \label{thm2}
Assume that $T> \operatorname{diam}_{b^{-1}g }M$ and the above geometric assumption holds. Let $\Lambda^\textrm{W}_{\beta_1}$ and $\Lambda^\textrm{W}_{\beta_2}$ denote the DtN maps for equation \eqref{JMGT2} with coefficients $\beta_1,\beta_2 \in C^{m}(\overline M)$ respectively and assume that 
\begin{equation}
\beta_j=0 \quad \text{on } \partial M, \quad j=1,2.
\end{equation}
If there exists $\delta>0$ such that
\begin{equation*}
\Lambda^{\mathrm W}_{\beta_1}(f)=\Lambda^{\mathrm W}_{\beta_2}(f)
\qquad
 \textrm{for all } f \in A_m \textrm{ with } (0,0,0,f) \in \mathcal{N}_{\delta}
 \end{equation*}
then one has 
\begin{equation}
\beta_1=\beta_2  \quad \text{in } M.
\end{equation}
\end{theorem}

For Kuznetsov-type nonlinearity, we have
\begin{theorem} \label{thm3}
Assume that $T> \operatorname{diam}_{b^{-1}g }M$ and the above geometric assumption holds. Let $\Lambda^\textrm{K}_{\beta_1,\kappa_1}$ and $\Lambda^\textrm{K}_{\beta_2,\kappa_2}$ denote the corresponding DtN maps for equation \eqref{JMGT} with coefficients $(\beta_1,\kappa_1),(\beta_2,\kappa_2) \in C^{m}(\overline M) \times C^{m}(\overline M)$ and assume that 
\begin{equation}
\beta_j=\kappa_j=0 \quad \text{on } \partial M, \quad j=1,2.
\end{equation}
If there exists $\delta>0$ such that
\begin{equation*}
\Lambda^\textrm{K}_{\beta_1,\kappa_1}(f)=\Lambda^\textrm{K}_{\beta_2,\kappa_2}(f) \qquad \textrm{for all } f \in A_{m}  \ \textrm{with } (0,0,0,f) \in \mathcal{N}_{\delta},
\end{equation*}
then one has 
\begin{equation}
\beta_1=\beta_2,\ \kappa_1=\kappa_2   \quad \textrm{in } M .
\end{equation}
\end{theorem}

\subsection{Previous Literature and Related Works}
Inverse problems for nonlinear hyperbolic equations have been extensively studied since the pioneering work \cite{kurylev_inverse_2018}. The paper \cite{kurylev_inverse_2018} observed that nonlinearity can be used as a powerful tool in inverse problems for nonlinear wave equations and showed that local measurements for the scalar wave equation with a quadratic nonlinearity determine the conformal class of a globally hyperbolic four-dimensional Lorentzian manifold. Their approach, now known as the higher order linearization method, was extended to elliptic equations in \cite{feizmohammadi_inverse_2020,lassas_inverse_2021} for full data and in \cite{krupchyk_remark_2019,lassas_partial_2020} for partial data. For inverse boundary value problems in the hyperbolic setting, the method was further developed in \cite{feizmohammadi_recovery_2022,hintz_dirichlet--neumann_2022,hintz_inverse_2022,lassas_gaussian_2025,lassas_stability_2025,lassas_uniqueness_2022,oksanen_inverse_2024,qiu_uniqueness_2025,uhlmann_inverse_2021,wang_inverse_2019}. We refer the reader to \cite{oksanen_inverse_2024} for the unified approach treating general real principal type differential operators.

We also note that inverse problems arising in nonlinear acoustics, such as ultrasound imaging and infrasonic wave propagation, have been considered in recent years. The work \cite{acosta_nonlinear_2022} showed that the DtN map uniquely determines the nonlinearity coefficient in the Westervelt equation without damping arising in ultrasound imaging. Further related studies on inverse problems for the Westervelt equation include \cite{acosta_simultaneous_2025,Kaltenbacher_2023,li_inverse_2024,uhlmann_inverse_2023}. Furthermore, H\"older stability was proved recently in \cite{wendels_stable_2025}. For the nonlinear progressive wave equation arising in infrasonic wave propagation, the determination of multiple unknowns has been studied in both Minkowski space and Lorentzian manifolds in \cite{jiang_inverse_2025,jiang_inverse_2025-1}.

Among these, the work most closely related to ours is \cite{fu_inverse_2024}, where the authors studied the well-posedness for small initial and boundary data and the unique recovery of single Westervelt-type nonlinear coefficient from all boundary measurements or the input-output map in Minkowski space. They also obtained logarithmic stability in the corresponding setting in \cite{fu_stability_2025-1}. The JMGT equation, arising in wave propagation in viscous thermally relaxing fluids, has recently received significant attention in the literature on inverse problems. The paper \cite{kaltenbacher_acoustic_2025} proves local uniqueness of the spatial acoustic nonlinearity parameter in the JMGT frequency-domain model from boundary measurements and outlines a convergent regularized Newton reconstruction. We refer the reader to \cite{fu_partial_2026,fu_stability_2025,fu_calderon_2025,kaltenbacher_imaging_2025} for some recent works on inverse problems for the JMGT and MGT equations. 

In contrast to \cite{fu_inverse_2024}, we consider inverse problems for a more involved model with two unknown nonlinear coefficients and also extend the analysis to compact three-dimensional Riemannian manifolds under suitable geometric assumptions. We show that the associated Dirichlet-to-Neumann map uniquely determines these coefficients both in Euclidean domains and on compact Riemannian manifolds. Our approach combines the second order linearization method with the construction of geometric optics solutions in the Euclidean setting and Gaussian beam solutions on manifolds, reducing the inverse problem to the injectivity of certain geodesic ray transforms. To the best of our knowledge, this appears to be the first uniqueness result for recovering both quadratic nonlinear coefficients in the JMGT equation from the Dirichlet-to-Neumann map.

The paper is organized as follows. In Section \ref{wellposedness}, we establish the local well-posedness for the forward problem. Section \ref{Euclidean} presents the construction of geometric optics solutions of the linear MGT equation and the proof of Theorem \ref{thm1} in the Euclidean setting. In Section \ref{Gauss}, we construct the Gaussian beam solution for the linear MGT equation in the geometric setting. Section \ref{geometric} is devoted to the proofs of Theorem \ref{thm2} and Theorem \ref{thm3} under the Geometric Assumption. Section \ref{conclusion} contains the conclusion and future work.

\section{Local Well-Posedness of JMGT equation} \label{wellposedness}


The aim of this section is to prove Proposition \ref{propo1}, which establishes the well-posedness of the initial boundary value problem \eqref{JMGT} for small initial and boundary data. 

We begin with the following linear MGT equation:
\begin{equation}\label{MGT}
\left\{
\begin{aligned}
&\partial_t^3v + \alpha \partial_{t}^2 v  - b\Delta_g \partial_t v-c^2\Delta_g v = F(t,x)
    && \text{in } \mathcal{M}, \\
&v= h(t,x)
    && \text{on } \Gamma, \\
&v(0,x)=v_0(x), \partial_t v(0,x) = v_1(x), \partial_t^2 v(0,x)=v_2(x) 
    && \text{in }  M.
\end{aligned}
\right.
\end{equation}
We introduce the following MGT  equation compatibility conditions on the boundary data and initial data 
 $(v_0,v_1,v_2,h)$ in \eqref{MGT}.
\begin{definition}
On the boundary $\partial M$, if the following relations hold
\begin{equation}\label{compa}
\left\{
\begin{aligned}
&\partial_t^kh(0,x)=v_k(x)
     \quad  \text{on } \partial M \text{ for } k=0,1,2 ,\\
&\partial_t^{k}h(0,x)=-\alpha \partial_t^{k-1}v(0,x)+b\Delta_g \partial_t^{k-2}v(0,x)+c^2\Delta_g \partial_t^{k-3}v(0,x)+\partial_t^{k-3} F(0,x) \\
    & \quad \quad  \quad  \quad  \quad  \quad  \quad  \quad  \quad   \quad  \text{on } \partial M \text{ for } 3 \leq k \leq m+1, \\
\end{aligned}
\right.
\end{equation}
the quadruple $(v_0,v_1,v_2,h)$ is said to satisfy the MGT equation compatibility conditions up to order $m+1$.
\end{definition}

\begin{lemma} \label{2}
Suppose that $F \in E^m$ and $(v_0,v_1,v_2,h) \in H^{m+2}(M)\times H^{m+1}(M)\times H^m(M) \times A_m$ satisfy the compatibility conditions    \eqref{compa}      up to order $m+1$. Then linear system \eqref{MGT} admits a unique solution $v \in E^{m+2}$ such that $\partial_{\nu}v \in N_{m+1}$ and
\begin{align*}
& \left\|v \right\|_{E^{m+2}}+\left\| \partial_{\nu}v  \right\|_{N_{m+1}}\leq \\ &C   
\left( \left\|v_0\right\|_{H^{m+2}(M)}+ \left\|v_1\right\|_{H^{m+1}(M)}+\left\|v_2\right\|_{H^{m}(M)}+\left\|F\right\|_{E^m}+\left\|h\right\|_{H^{m+2}(\Gamma)}   \right).
\end{align*}
\end{lemma}
\begin{proof}
This is a generalization of \cite[Theorem 3.1]{fu_inverse_2024}. For brevity, we focus on the case $m=0$ here, while higher-order regularity can be obtained by the bootstrap arguments.

To deal with the non-homogeneous boundary condition, we construct a lifting function $\tilde{h}(t,x)$ as the solution to the wave equation:
\begin{equation}\label{eq:lifting_h}
\left\{
\begin{aligned}
&\partial_t^2 \tilde{h} - b \Delta_g \tilde{h}= 0 \quad &&\text{in } \mathcal{M}, \\
&\tilde{h} = h \quad &&\text{on } \Gamma, \\
&\tilde{h}(0,x) = \tilde{h}_0,\ \partial_t \tilde{h}(0,x) = \tilde{h}_1 \quad&&\text{on } M,
\end{aligned}
\right.
\end{equation}
where $\tilde{h}_0 \in H^2(M)$ and $\tilde{h}_1 \in H^1(M)$ are chosen to satisfy the compatibility conditions on the boundary. Since $h \in A_0 \subset H^2(\Gamma)$, \cite[Theorem 2.2]{lasiecka1986non} guarantees that \eqref{eq:lifting_h} admits a unique solution
\[
\tilde{h} \in C([0,T];H^2(M))\cap C^1([0,T];H^1(M)) \cap C^2([0,T];L^2(M)).
\]
Define $\tilde{v}(t,x) = v(t,x) - \tilde{h}(t,x)$. Substituting $v$ into \eqref{MGT} yields the following problem for $\tilde{v}$:
\begin{equation}\label{MGT2}
\left\{
\begin{aligned}
&\partial_t^3 \tilde{v} + \alpha \partial_t^2 \tilde{v} - b \Delta_g \partial_t \tilde{v} - c^2 \Delta_g \tilde{v} = \tilde{F}(t,x) \quad &&\text{in } \mathcal{M}, \\
&\tilde{v} = 0 \quad &&\text{on } \Gamma, \\
&\tilde{v}(0,x) = \tilde{v}_0,\ \partial_t \tilde{v}(0,x) = \tilde{v}_1,\ \partial_t^2 \tilde{v}(0,x) = \tilde{v}_2 \quad&&\text{on } M,
\end{aligned}
\right.
\end{equation}
where
\[
\begin{aligned}
\tilde{F} &= F - \gamma \partial_t^2 \tilde{h} \in C([0,T];L^2(M)), \\
\tilde{v}_0 &= v_0 - \tilde{h}_0, \quad \tilde{v}_1 = v_1 - \tilde{h}_1, \quad \tilde{v}_2 = v_2 - \partial_t^2 \tilde{h}(0,x).
\end{aligned}
\]
Define the conformal metric $g'=b(x)^{-1}g$. The Laplace–Beltrami operator transforms as
\[
b(x)\Delta_g \tilde{v} = \Delta_{g'} \tilde{v} + P_1\tilde{v}.
\]
where $P_1u=\frac{n-2}{2} \langle\mathrm{d} b, \mathrm{d} u\rangle_g$ is a first-order differential operator. Let $\mathcal{A}_{g'}=-\Delta_{g'}$ denote the positive self-adjoint operator with compact resolvent on $L^2(M,\mathrm{d}V_{g'})$ and has the domain $\mathcal{D}(\mathcal{A}_{g'})=H^2(M)\cap H_0^1(M)$. The fractional power operator $\mathcal{A}_{g'}^{1/2}$ has domain $\mathcal{D}(\mathcal{A}_{g'}^{1/2})=H_0^1(M)$. We consider the high-regularity state space $\mathcal{H}=\mathcal{D}(\mathcal{A}_{g'})\times \mathcal{D}(\mathcal{A}_{g'}^{1/2})\times L^2(M,\mathrm{d}V_{g'})$. Then the MGT equation can be rewritten as:
\begin{equation}\label{eq:MGT}
    \partial_t^3 \tilde{v}+\alpha\partial_t^2 \tilde{v} + \omega(x)\mathcal{A}_{g'}\tilde{v} + \mathcal{A}_{g'}\partial_t \tilde{v} - \omega(x)P_1\tilde{v} - P_1\partial_t\tilde{v}=\tilde{F} \quad \text{in }\mathcal{M}, 
\end{equation}
where $\omega(x)= \frac{c^2(x)}{b(x)}$.\\
We introduce $z= \tilde{v}_t + \omega(x)\tilde{v}$ and $y=\mathcal{A}_{g'}\tilde{v} + z_t$, then construct the new state vector $Y=(y,\mathcal{A}_{g'}^{1/2}z,z_t)^T$ and the product Hilbert space $\mathcal{Z}=L^2(M,\mathrm{d}V_{g'})\times L^2(M,\mathrm{d}V_{g'})\times L^2(M,\mathrm{d}V_{g'})$ with the weighted inner product 
\begin{align*}
    \langle Y,\hat{Y} \rangle_{\mathcal{Z}} = \langle y,\hat{y} \rangle_{L^2(g')} + \langle \mathcal{A}_{g'}^{1/2}z,\mathcal{A}_{g'}^{1/2}\hat{z} \rangle_{L^2(g')} + \langle z_t, \hat{z}_t \rangle_{L^2(g')}.
\end{align*}
A simple calculation shows that \eqref{eq:MGT} can be rewritten as a first-order operator evolution equation $\frac{dY(t)}{dt}=DY(t)+G(t)$ on $\mathcal{Z}$, where $G(t)$ is the source term vector, and the matrix operator $D$ decomposes as $D=D_0+K_1$:
\[
\begin{aligned}
D_0= B_1 +P_B=
\mathcal{A}_{g'}^{1/2}\left(\begin{array}{ccc}0&0&0\\0&0&I\\0&-I&0\end{array}\right) +
\left(
\begin{array}{ccc}
 -\omega(x) I & 0 & \left(\omega(x)-\gamma(x)\right) I \\
 0 & 0 & 0 \\
 0 & 0 & -\gamma(x) I
\end{array}
\right) .
\end{aligned}
\]
Here, $K_1$ collects all the residual terms arising from the lower-order spatial operators, and $\mathcal{D}(D_0) = L^2(M,\mathrm{d}V_{g'})\times \mathcal{D}(\mathcal{A}^{1/2}_{g'}) \times \mathcal{D}(\mathcal{A}^{1/2}_{g'})$.\\
By the same argument as in the proof of \cite[Theorem 1.2]{kaltenbacher2011wellposedness}, the principal part $B_1$ with the domain $\mathcal{D}(B_1)=\mathcal{D}(D_0)$ is a skew-adjoint operator on $\mathcal{Z}$. Since $P_B$ and $K_1$ are bounded on $\mathcal{Z}$, $D=B_1+P_B+K_1$ is a bounded perturbation of $B_1$
 and also generates a strongly continuous group on $\mathcal{Z}$ via the bounded perturbation theorem for semigroups. Since the coefficients $b(x),c(x)>0$ on the compact manifold $M$, the multiplier $\omega$ is strictly positive and bounded away from zero. This guarantees that the transformation from the original state variables $\tilde{V}=(\tilde{v},\tilde{v}_t,\tilde{v}_{tt})^T \in \mathcal{H}$ to $Y\in \mathcal{Z}$ is a continuous isomorphism with a bounded inverse. Therefore, the first-order matrix operator corresponding to \eqref{eq:MGT} is topologically equivalent to D on $\mathcal{Z}$, and thus it generates a strongly continuous group on $\mathcal{H}$.\\
By Duhamel's principle, we can prove that equation \eqref{MGT2} admits a unique solution $\tilde{V}(t)=(\tilde{v}(t),\tilde{v}_t(t),\tilde{v}_{tt}(t))^T \in C([0,T];\mathcal{H})$. Thus, \eqref{MGT} admits a unique solution
\[
v \in C([0,T];H^2(M))\cap C^1([0,T];H^1(M)) \cap C^2([0,T];L^2(M)).
\]

Respectively multiplying both sides the first equation of \eqref{MGT} by $v_{tt}$ and $v_t$, integrating over $(0,t) \times M$, applying the Green formula, we have
\begin{align*}
&\frac12 \int_M \bigl(v_{tt}^2 + b|\nabla_g v_t|_g^2\bigr)\ \mathrm{d}V_g
+ \int_0^t \int_M \alpha v_{tt}^2 \mathrm{d}V_g\mathrm{d}l
+ \int_M c^2\langle \nabla_g v, \nabla_g v_t \rangle_g \mathrm{d}V_g \\
& = \frac12 \int_M \bigl(v_2^2 + b|\nabla_g v_1|_g^2\bigr)\mathrm{d}V_g
+ \int_M c^2\langle \nabla_g v_1, \nabla_g v_0 \rangle_g \mathrm{d}V_g
+ \int_0^t \int_M c^2|\nabla_g v_t|_g^2 \mathrm{d}V_g\mathrm{d}l \\
&+ \int_0^t \int_M F v_{tt} \mathrm{d}V_g\mathrm{d}l
+ \int_0^t \int_{\partial M} h_{tt} \bigl(c^2 \partial_\nu v + b \partial_\nu v_t\bigr)\mathrm{d}S_g \mathrm{d}l \\
&- \int_0^t \int_M v_{tt}\left(\langle \nabla_g v_t, \nabla_g b \rangle_g + \langle \nabla_g v, \nabla_g c^2 \rangle_g \right) \mathrm{d}V_g\mathrm{d}l,
\end{align*}
and
\begin{align*}
&\frac12 \int_M \bigl(\alpha v_t^2 + c^2|\nabla_g v|_g^2\bigr)\mathrm{d}V_g +\int_M v_tv_{tt}\mathrm{d}V_g
+ \int_0^t \int_M b|\nabla_g v_t|_g^2 \mathrm{d}V_g\mathrm{d}l \\
& = \int_0^t \int_M  v_{tt}^2 \ \mathrm{d}V_g\mathrm{d}l + \frac12 \int_M \bigl(\alpha v_1^2+c^2|\nabla_g v_0|_g^2 + 2v_1v_2\bigr)\mathrm{d}V_g + \int_0^t \int_M F v_t \mathrm{d}V_g\mathrm{d}l \\
&+ \int_0^t \int_{\partial M} h_t \bigl(c^2 \partial_\nu v + b \partial_\nu v_t\bigr)\mathrm{d}S_g \mathrm{d}l - \int_0^t \int_M v_t\left(\langle \nabla_g v_t, \nabla_g b \rangle_g + \langle \nabla_g v, \nabla_g c^2 \rangle_g \right) \mathrm{d}V_g\mathrm{d}l,
\end{align*}
where lower-order integral terms generated by variable coefficients appear on the right-hand side. These terms can be absorbed by the Gronwall inequality with the help of the Cauchy-Schwartz inequality. \\
Then multiplying equation \eqref{MGT} by $L_gv=v_{tt}-b\Delta_gv$ and integrating over $(0,t) \times M$, we have
\begin{align*}
&\frac{1}{2} \int_M \left|L_g v(t)\right|^2 \mathrm{d}V_g + \frac{1}{2} \int_0^t \int_M \omega\left|L_g v\right|^2 \mathrm{d}V_g \mathrm{d}l\\
&\leq \frac{1}{2} \int_M \left|L_g v(0)\right|^2 \mathrm{d}V_g + C \int_0^t \int_M \left(|F|^2 + v_{tt}^2\right) \mathrm{d}V_g \mathrm{d}l.
\end{align*}
Together with above three equalities, we arrive at the following estimate:
\begin{align}
    &\frac12 \int_M \bigl(v_{tt}^2 + |\nabla_g v_t|_g^2 + v_t^2 + |\nabla_g v|_g^2 + |\Delta_g v|^2\bigr)\mathrm{d}V_g \nonumber \\
& \leq C\Bigl( \|v_0\|_{H^2(M)}^2 + \|v_1\|_{H^1(M)}^2 + \|v_2\|_{L^2(M)}^2 \Bigr)
+ C \int_0^t \int_M \bigl(|\nabla_g v_t|_g^2 + v_{tt}^2 + v_t^2 + |\nabla_g v|_g^2\bigr)\mathrm{d}V_g\mathrm{d}l \nonumber \\
& + C\Bigl( \|F\|_{L^2(\mathcal{M})}^2 + \|h\|_{H^2([0,T];L^2(\partial M))}^2 + \|\partial_\nu v\|_{H^1([0,T];L^2(\partial M))}^2 \Bigr).\label{esm1}
\end{align}
We next estimate the term $\|\partial_\nu v\|_{H^1([0,T];L^2(\partial M))}^2$ appearing in the last term of \eqref{esm1}. Let $H$ be a smooth vector field on $\overline{M}$ such that $H|_{\partial M} = \nu$. Here, we denote by $DH$ the covariant derivative tensor of $H$ with respect to $g$, defined by $DH(X,Y)=\langle\nabla_XH,Y\rangle_g$ for any vector fields $X,Y$, and by $div_gH = tr_g(DH)$ the standard divergence on $(M,g)$.\\
We calculate $2(Hv)L_gv$ and integrate it over $(0,t) \times M$, via the divergence theorem, we obtain
\begin{align}
    \int_0^t \int_{\partial M} b|\partial_{\nu}v|_g^2\mathrm{d}S_g\mathrm{d}l &= -2 \int_0^t \int_M(Hv)L_gv \mathrm{d}V_g\mathrm{d}l + 2\int_M v_t Hv \mathrm{d}V_g - 2\int_M v_1Hv_0 \mathrm{d}V_g \nonumber \\
    &+\int_0^t\int_M \left(2b DH(\nabla_g v,\nabla_g v) -(v_t^2 - b|\nabla_g v|_g^2) \text{div}_gH  \right)\mathrm{d}V_g\mathrm{d}l \nonumber \\
    &- \int_0^t\int_M \langle \nabla_g b, 2(Hv)\nabla_g v -  |\nabla_g v|_g^2 H \rangle_g\mathrm{d}V_g\mathrm{d}l  \nonumber \\
    &+ \int_0^t\int_{\partial M}(b|\nabla_{\mu} h|_g^2 - h_t^2) \mathrm{d}S_g\mathrm{d}l.\label{partialv1}
\end{align}
where $\nabla_{\mu}v$ denotes the tangential gradient of $v$ along $\partial M$ with respect to the metric $g$, and we apply the orthogonal decomposition of the gradient on the boundary: $|\nabla_g v|_g^2 = |\nabla_{\mu} v|_g^2 + |\partial_{\nu} v|_g^2$.\\
Multiplying $L_g v_t$ by $Hv_t$, by the same argument, we have
\begin{align}
    \int_0^t \int_{\partial M} b|\partial_{\nu}v_t|_g^2\mathrm{d}S_g\mathrm{d}l &= -2 \int_0^t \int_M(Hv_t)L_gv_t \mathrm{d}V_g\mathrm{d}l + 2\int_M v_{tt} Hv_t \mathrm{d}V_g - 2\int_M v_2Hv_1 \mathrm{d}V_g \nonumber \\
    &+\int_0^t\int_M \left(2b DH(\nabla_g v_t,\nabla_g v_t) -(v_{tt}^2 - b|\nabla_g v_t|_g^2) \text{div}_gH  \right)\mathrm{d}V_g\mathrm{d}l \nonumber \\
    &- \int_0^t\int_M \langle \nabla_g b, 2(Hv_t)\nabla_g v_t -  |\nabla_g v_t|_g^2 H \rangle_g\mathrm{d}V_g\mathrm{d}l \nonumber \\
    &+ \int_0^t\int_{\partial M}(b|\nabla_{\mu} h_t|_g^2 - h_{tt}^2) \mathrm{d}S_g\mathrm{d}l.\label{partialv2}
\end{align}
Notice that the second-last terms on the right-hand sides of \eqref{partialv1}, \eqref{partialv2} are lower-order integral terms arising from the variable coefficient $b$, which can be bounded by the integrals of $|\nabla_g v|_g^2, |\nabla_g v_t|_g^2$ over $(0,t) \times M$. In particular, this implies that $\partial_{\nu} v \in N_1$.\\
It follows from \eqref{esm1}, \eqref{partialv1} and \eqref{partialv2} that
\begin{align}
        &\int_M \bigl(v_{tt}^2 + |\nabla_g v_t|_g^2 + v_t^2 + |\nabla_g v|_g^2 + |\Delta_g v|^2\bigr)\mathrm{d}V_g \nonumber \\
& \leq C\Bigl( \|v_0\|_{H^2(M)}^2 + \|v_1\|_{H^1(M)}^2 + \|v_2\|_{L^2(M)}^2 \Bigr)  \nonumber \\
&+ C \int_0^t \int_M \bigl(|\nabla_g v_t|_g^2 + v_{tt}^2 + v_t^2 + |\nabla_g v|_g^2 + |\Delta_g v|^2\bigr)\mathrm{d}V_g\mathrm{d}l \nonumber \\
& + C\Bigl( \|F\|_{L^2(\mathcal{M})}^2 + \|h\|_{H^2([0,T];L^2(\partial M))}^2 + \|h\|_{H^1([0,T];H^1(\partial M))}^2 \Bigr).\label{esm2}
\end{align}
Applying the Gronwall inequality and combining \eqref{partialv2} and \eqref{esm2}, yields
\begin{equation*}
\begin{aligned}
\|v\|_{E^2} + \|\partial_{\nu}v\|_{H^1([0,T];L^2(\partial M))}
&\le C\Bigl(
\|v_0\|_{H^2(M)} + \|v_1\|_{H^1(M)} + \|v_2\|_{L^2(M)} \\
&\qquad\quad
+ \|F\|_{L^2(\mathcal{M})} + \|h\|_{H^2(\Gamma)}
\Bigr).
\end{aligned}
\end{equation*}
Then, the proof of the case $m=0$ is completed.
\end{proof}
In light of the above, we can now proceed to the proof of Proposition \ref{propo1}.
\begin{proof}
Since $(v_0,v_1,v_2,f) \in H^{m+2}(M)\times H^{m+1}(M)\times H^m(M) \times A_{m}$ satisfies the compatibility conditions \eqref{compa},
by  Lemma \ref{2}, let $u_0$ be the solution to 
\begin{equation}\label{JMGTlinear}
\left\{ 
\begin{aligned}
&\partial_t^3u_0 + \alpha \partial_{t}^2 u_0  - b\Delta_g \partial_t u_0-c^2\Delta_g u_0 = 0
    && \text{in } \mathcal{M}, \\
&u_0= f(t,x)
    && \text{on } \Gamma, \\
&u_0(0,x)=v_0(x), \partial_t u_0(0,x) = v_1(x), \partial_t^2 u_0(0,x)=v_2(x) 
    && \text{in }  M,
\end{aligned}
\right.
\end{equation}
then $u_0$ satisfies the estimate
\begin{align}
\|u_0\|_{E^{m+2}}+\|\partial_{\nu}u_0\|_{N_{m+1}}
&\le C\Big(
\|v_0\|_{H^{m+2}(M)}+\|v_1\|_{H^{m+1}(M)}
+\|v_2\|_{H^{m}(M)}+\|f\|_{H^{m+2}(\Gamma)}
\Big) \notag \\
&\le C\delta . \label{estimate1}
\end{align}
Given $w \in B_{\delta}:=\left\{u \in E^{m+2}: \|u\|_{E^{m+2}} \leq \delta
\right\}$. Let $\tilde u$ be the solution of 
\begin{equation}\label{fixpoint}
\left\{
\begin{aligned}
&\partial_t^3 \tilde u + \alpha \partial_{t}^2 \tilde u  - b\Delta_g \partial_t \tilde u-c^2\Delta_g \tilde u = \beta \partial_{t}(\partial_tw+\partial_t u_0)^2 + \kappa\partial_t|\nabla_g w+\nabla_g u_0|_g^2 
    && \text{in } \mathcal{M}, \\
&\tilde u= 0
    && \text{on } \Gamma, \\
&\tilde u(0,x)=0, \partial_t \tilde u(0,x) = 0, \partial_t^2 \tilde u(0,x)=0 
    && \text{in }  M.
\end{aligned}
\right.
\end{equation}
Since $E^m$ is an algebra, $(\beta,\kappa) \in C^m( \overline M ) \times C^m ( \overline M )$ and $\beta=0, \ \kappa=0$ on $\partial M$, one can show that 
\begin{align*}
&L(t,x):=\beta \partial_{t}(\partial_tw+\partial_t u_0)^2 + \kappa\partial_t|\nabla_g w+\nabla_g u_0|_g^2  \in E^m, \\
&\partial_t^kL(t,x)=0 \text{ on } \left\{t=0 \right\} \times \partial M \text{ for } k=0,1,2,\cdots,m-2 ,
\end{align*}
then $(0,0,0,0)$ satisfies the compatibility conditions \eqref{compa}.
Moreover, from Lemma \ref{2}, we have 
\begin{align}
&\| \tilde u \|_{E^{m+2}}+ \| \partial_{\nu}\tilde u \|_{N_{m+1}} \leq C  \| \beta \partial_{t}(\partial_tw+\partial_t u_0)^2 + \kappa\partial_t|\nabla_g w+\nabla_g u_0|_g^2 \|_{E^m} \nonumber  \\
&\leq C (\| \left(\partial_t w+\partial_t u_0\right)^2  \|_{E^{m+1}} + \| \left|\nabla_g w+\nabla_g u_0\right|_g^2  \|_{E^{m+1}}) \nonumber \\
& \leq C (\| w\|_{E^{m+2}}^2+\| u_0  \|_{E^{m+2}}^2) \leq C \delta^2. \label{estimate2}
\end{align} 
For $\delta$ small enough, we have $\tilde u \in B_{\delta}$. Then we define a map
\begin{equation}
\mathcal{G}: B_{\delta}\to B_{\delta}, \quad  \mathcal {G}(w)=\tilde u.
\end{equation}
Next let $w_j \in B_{\delta}$ for $j=1,2$, and $\tilde u_j$ be the solution to \eqref{fixpoint} with $w$ replaced by $w_j$. Then we have their difference $V:=\tilde u_1 -\tilde u_2$ satisfying the equation
\begin{equation*}
\left\{
\begin{aligned}
&\partial_t^3 V + \alpha \partial_{t}^2 V  - b\Delta_g \partial_t V-c^2\Delta_g V  \\
&= \beta\partial_t( \partial_t(w_1-w_2)\partial_t(w_1+w_2+2u_0)) +\kappa\partial_t(\langle \nabla_g(w_1-w_2) ,\nabla_g(w_1+w_2+2u_0)\rangle_g) 
    && \text{in } \mathcal{M}, \\
&V= 0
    && \text{on } \Gamma, \\
&V(0,x)=0, \partial_t V(0,x) = 0, \partial_t^2 V(0,x)=0 
    && \text{in }  M.
\end{aligned}
\right.
\end{equation*}
Therefore, by using Lemma \ref{2} we obtain 
\begin{align*}
&\| V \|_{E^{m+2}}+ \| \partial_{\nu}V \|_{N_{m+1}}  \\ 
&\leq C\| \beta\partial_t(\partial_t(w_1-w_2)\partial_t(w_1+w_2+2u_0)) +\kappa\partial_t(\langle \nabla_g(w_1-w_2) ,\nabla_g(w_1+w_2+2u_0)\rangle_g)  \|_{E^m} \\
&\leq C \left(\| \partial_t (w_1-w_2)\partial_t(w_1+w_2+2u_0)  \|_{E^{m+1}} + \| \langle \nabla_g(w_1-w_2) ,\nabla_g(w_1+w_2+2u_0)\rangle_g \|_{E^{m+1}}\right) \\
& \leq C \| w_1+w_2+2u_0\|_{E^{m+2}}\| w_1-w_2  \|_{E^{m+2}} \leq C \delta \|w_1-w_2\|_{E^{m+2}}.
\end{align*}
For $\delta$ small enough, it implies that $\mathcal{G}$ is a contraction on $B_{\delta}$. By the Banach fixed-point theorem, there exists a unique fixed point $\bar u$ which is the solution to \eqref{fixpoint} with $\tilde u$ and $w$ replaced by $\bar u$. Therefore $u:=\bar u+u_0$ is the solution to \eqref{JMGT} with Kuznetsov-type nonlinearity, and by combining \eqref{estimate1} and \eqref{estimate2}, we obtain that the estimate \eqref{estimate3} holds.
\end{proof}
The proof of Proposition \ref{propo2} is very similar to that of Proposition \ref{propo1}, so we omit the details. 

\begin{remark}
It is not hard to see that the solution map $(v_0,v_1,v_2,f) \to u$ is $C^{\infty}$ Fr\'{e}chet differentiable in two different nonlinearity cases.
\end{remark}

\section{The Euclidean Case} \label{Euclidean}
In this section, we present second order linearization method, the construction of geometric optics solutions of
the linear MGT equation and the proof of Theorem \ref{thm1}. In this section, we set
\begin{equation*}
Q:=(0,T) \times \Omega.
\end{equation*}
\subsection{Second Order Linearization}

We will use a second order linearization of the DtN map to study our inverse problem. This higher order linearization technique has been extensively used in the literature, see \cite{acosta_nonlinear_2022,hintz_inverse_2022,uhlmann_inverse_2021}. In the zero initial data case, we choose boundary data $f_0,f_1,f_2 \in A_m$ such that $f_1$ and $f_2$ vanish near $\{ t=0 \}$ and $f_0$ vanishes near $\{ t=T \}$.  
We choose $\epsilon_1$ and $\epsilon_2$ sufficiently small so that $u$ is the unique solution to \eqref{JMGT} with  boundary data $f = \epsilon_1 f_1 + \epsilon_2 f_2$ on $\Gamma$,  and define
\begin{equation}
 u_i= \frac{\partial u} {\partial \epsilon_i}\Bigg|_{\epsilon_1=\epsilon_2=0}, \ w = \frac{\partial^2 u}{\partial \epsilon_1 \partial \epsilon_2} \Bigg|_{\epsilon_1=\epsilon_2=0}.
\end{equation}
Then $w$ satisfies
\begin{equation}\label{eq:second-order}
\left\{
\begin{aligned}
&\partial_{t}^3 w +\alpha \partial_t^2 w - b \Delta \partial_t w -c^2 \Delta w= 2 \beta\partial_{t} (\partial_t u_{1} \partial_t u_{2}) +2\kappa \partial_t(\nabla u_1\cdot\nabla u_2) 
 && \text{in }Q,\\ 
&w = 0 
    && \text{on } \Gamma, \\
&w = \partial_t w=\partial_t^2 w = 0 
    && \text{on } \{t=0\} \times \Omega,
\end{aligned}
\right.
\end{equation}
where, for $i=1,2$, $u_i$ is the solution to the linear MGT equation
\begin{equation*}\label{forward}
\left\{
\begin{aligned}
&\partial_{t}^3 u_i +\alpha \partial_t^2 u_i - b \Delta \partial_t u_i -c^2 \Delta u_i= 0
 && \text{in }Q,\\ 
&u_i = f_i 
    && \text{on } \Gamma, \\
&u_i = \partial_t u_i=\partial_t^2 u_i = 0 
    && \text{on } \{t=0\} \times \Omega.
\end{aligned}
\right.
\end{equation*}
Additionally, let $u_{0}$ solve the backward problem:
\begin{equation}\label{eq:backward}
\left\{
\begin{aligned}
&-\partial_{t}^3 u_0 +\alpha \partial_t^2 u_0 + b \Delta \partial_t u_0 -c^2 \Delta u_0= 0
 && \text{in }Q,\\ 
&u_0 = f_0 
    && \text{on } \Gamma, \\
&u_0 = \partial_t u_0=\partial_t^2 u_0 = 0 
    && \text{on } \{t=T\} \times \Omega.
\end{aligned}
\right.
\end{equation}
Observe that $\frac{\partial^2}{\partial \epsilon_1 \partial \epsilon_2} \Lambda^\textrm{K}_{\beta,\kappa}(\epsilon_1 f_1 +\epsilon_2 f_2)|_{\epsilon_1=\epsilon_2=0} =\partial_{\nu} (c^2w+b\partial_t w)|_{\Gamma}$.
Integration by parts yields
\begin{align}\label{eq:ibp}
&\int_{\Gamma} \frac{\partial^2}{\partial \epsilon_1 \partial \epsilon_2} \Lambda^\textrm{K}_{\beta,\kappa}(\epsilon_1 f_1 +\epsilon_2 f_2)\Big|_{\epsilon_1=\epsilon_2=0}  f_0  \mathrm{d}s  \mathrm{d}t \nonumber \\
&=\int_{\Gamma} \partial_{\nu} (c^2w+b\partial_t w) f_0  \mathrm{d}s  \mathrm{d}t \nonumber \\
&=\int_{Q} \nabla\cdot (c^2\nabla w+b\nabla \partial_t w) u_0  \mathrm{d}x  \mathrm{d}t+ \int_{Q} (c^2\nabla w+b\nabla \partial_t w)\cdot\nabla u_0     \mathrm{d}x \mathrm{d}t \nonumber \\
&= \int_Q (\partial_t^3 w+\alpha \partial_t^2w-2 \beta\partial_{t} (\partial_t u_{1} \partial_t u_{2}) -2\kappa \partial_t(\nabla u_1\cdot\nabla u_2) )u_0-c^2 w\Delta u_0 - b\partial_t w \Delta u_0 \mathrm{d}x \mathrm{d}t   \nonumber \\
&= \int_Q w(-\partial_t^3u_0+\alpha \partial_t^2 u_0+b\Delta \partial_tu_0-c^2\Delta u_0)+ 2\beta \partial_tu_1 \partial_tu_2  \partial_tu_{0} +2\kappa\nabla u_1\cdot \nabla u_2\partial_t u_0  \mathrm{d}x  \mathrm{d}t  \nonumber \\
&= \int_Q 2\beta \partial_tu_1 \partial_tu_2  \partial_tu_{0} +2\kappa\nabla u_1\cdot \nabla u_2 \partial_t u_0  \mathrm{d}x  \mathrm{d}t .
\end{align}
It remains to recover the parameters $\beta(x)$ and $\kappa(x)$ from the integral identity \eqref{eq:ibp}. Applying \eqref{eq:ibp} with $(\beta,\kappa)=(\beta_j,\kappa_j)$ for $j=1,2$ and subtracting the resulting two identities, we obtain
\begin{equation}\label{identity}
 \int_Q\beta \partial_tu_1 \partial_tu_2  \partial_tu_{0} +\kappa\nabla u_1\cdot \nabla u_2 \partial_t u_0  \mathrm{d}x  \mathrm{d}t=0, 
\end{equation}
where $\beta$ denotes $\beta_1-\beta_2$ and $\kappa$ denotes $\kappa_1-\kappa_2$.

\subsection{Construction of GO solutions}

To prove the theorem, we construct \textit{geometric optics} (GO) solutions  to the linear problem
\begin{equation}\label{linearPu}
\begin{aligned}
Pu :=\,& \partial_t^3 u+ \alpha\partial_t^2 u - b\Delta\partial_t u-c^2\Delta u=0 \quad \text{in } Q,\\
& u=\partial_t u=\partial_t^2u=0     \quad \text{on} \ \{ t=0\} \times\Omega,
\end{aligned}
\end{equation}
of the form
\begin{equation}\label{GOsolution}
        u(t,x)= \mathrm{e}^{\mathrm{i}\lambda \varphi(t,x)} a_{\lambda}(t,x)+R_\lambda(t,x),
\end{equation}
where $\lambda$ is a large parameter and $R_\lambda$ denotes the remainder term that tends to zero as $\lambda \to \infty$ in a suitable norm. 

We first construct an asymptotic solution of the form
\begin{equation*}
    u_{\lambda}= \mathrm{e}^{\mathrm{i}\lambda(\psi(x)+t)} a_{\lambda}(t,x)=\mathrm{e}^{\mathrm{i} \lambda( \frac{1}{\sqrt{b}}w\cdot x+t)}(a_0(t,x)+\lambda^{-1}a_1(t,x)) \text { in } Q,
\end{equation*}
where the amplitude has the form $a_{\lambda}(t,x) = a_0(t,x) + \lambda^{-1} a_1(t,x)$ and $\psi \in C^{\infty}\left(\overline{\Omega} \right)$ is a real-valued phase function.

Substituting this ansatz into the linear equation, we have 
\begin{equation}\label{Pu_lam}
    Pu_{\lambda} = \mathrm{e}^{\mathrm{i}\lambda(\psi(x)+t)}(\mathrm{i}\lambda^3 T_3a_{\lambda} +\lambda^2 T_2a_{\lambda} +\mathrm{i}\lambda T_1 a_{\lambda} +Pa_{\lambda}),
\end{equation}
where the operators $T_3$, $T_2$, and $T_1$ are defined by
\begin{align*}
    & T_3 a = (b|\nabla \psi|^2-1) a,\\
    & T_2 a = (b|\nabla \psi|^2-3)\partial_t a +2b \nabla\psi \cdot \nabla a + (b\Delta \psi+ c^2|\nabla \psi|^2 -\alpha)a, \\
    & T_1 a = 3\partial_t^2 a + (2 \alpha-b\Delta \psi) \partial_t a - 2\nabla \psi \cdot(b \nabla \partial_t a+ c^2 \nabla a) - b \Delta a - c^2\Delta \psi a.
\end{align*}
To make the leading-order term in \eqref{Pu_lam} vanish, we impose the eikonal equation
\begin{equation*}
    b \nabla\psi \cdot \nabla\psi=1,
\end{equation*}
which is solved by the linear phase function $\psi (x) = \frac{1}{\sqrt{b}} w \cdot x$, where $w \in \mathbb{S}^{n-1}$.\\
It follows that the amplitude satisfies
\begin{align}
    &T_2 a_0 =0 \label{eq:a0},\\
    &T_2 a_1 + i T_1 a_0 =0 \label{eq:a1}.
\end{align}
Recalling that $\gamma=\alpha-\frac{c^2}{b}$, we rewrite \eqref{eq:a0} and \eqref{eq:a1} as
\begin{align}
    &\partial_t a_0 - \sqrt{b}w \cdot \nabla a_0 + \frac{\gamma}{2} a_0 =0,\label{trans:a0}\\
    &\partial_t a_1 - \sqrt{b}w \cdot \nabla a_1 + \frac{\gamma}{2} a_1 = \tilde{\xi},\label{trans:a1}
\end{align}
where $2\tilde{\xi} = \mathrm{i}[3 \partial_t^2 a_0 + 2\alpha \partial_t a_0 -2\sqrt{b}w \cdot (\frac{c^2}{b} \nabla a_0 + \nabla \partial_t a_0) - b\Delta a_0]$.\\
We choose the principal and subprincipal amplitudes as
\begin{align}
    a_0(t,x) &= e^{-\frac{\gamma}{2}t}\,\phi_{\epsilon}(x+\sqrt{b}\,t w-p), \notag\\
    a_1(t,x) &= e^{-\frac{\gamma}{2}t}\,\phi_{\epsilon}(x+\sqrt{b}\,t w-p)
    + \int_0^t \tilde{\xi}\!\left(\tau, x+\sqrt{b}\,w(t-\tau)\right)
    e^{-\frac{\gamma}{2}(t-\tau)}\, d\tau, \label{eqa1}
\end{align}
where $p \notin \overline \Omega$, $\phi \in C_0^{\infty}(B_1 (0))$, $\int_{\mathbb{R}^n} \phi^3 \mathrm dx=1$ and $\phi_{\epsilon}(y)=\epsilon^{-\frac{n}{3}} \phi(\frac{y}{\epsilon})$. We choose $\epsilon>0$ sufficiently small such that, at $t=0$, the support of $\phi_\epsilon(\,\cdot-p)$ is contained in a sufficiently small neighborhood of $p$ and lies outside $\overline{\Omega}$. Then, by construction, $\partial_t^k a_0(0,x)=0$ in $\Omega$ for $k=0,1,2$. Then, by \eqref{eqa1}, we also have $\partial_t^k a_1(0,x)=0$ for $k=0,1,2$ in $\Omega$ and $a_{\lambda}$ satisfies the zero initial conditions.     

Then we define $F_{\lambda} = e^{i \lambda \psi} (i T_1 a_1 + Pa_0+\lambda^{-1} Pa_{1})$ and the remainder term satisfies the following linear system
\begin{equation} \label{eq:remainder}
\left\{
\begin{aligned}
&\partial_{t}^3 R_{\lambda} +\alpha \partial_t^2 R_{\lambda} - b \Delta \partial_t R_{\lambda} -c^2 \Delta R_{\lambda}= - \mathrm{e}^{\mathrm{i} \lambda t} F_{\lambda}
 && \text{in }Q,\\ 
&R_{\lambda} = 0 
    && \text{on } \Gamma, \\
&R_{\lambda} = \partial_t R_{\lambda}=\partial_t^2 R_{\lambda} = 0 
    && \text{on } \{t=0\} \times \Omega.
\end{aligned}
\right.
\end{equation} 
Similar to \cite[Theorem 4.1]{fu_inverse_2024}, we obtain the following result.
\begin{proposition}\label{propo3.1}
The equation \eqref{linearPu} admits a GO solution of the form \eqref{GOsolution}, where the remainder term $R_{\lambda}$ satisfies \eqref{eq:remainder}. Moreover, there exists a positive constant $C$ such that
\begin{align}\label{estimateR}
   & \lambda \left( \left\|R_{\lambda} \right\|_{L^2(Q)} + \left\|\partial_t R_{\lambda} \right\|_{L^2(Q)} + \left\| \nabla R_{\lambda} \right\|_{L^2(Q)} \right) + \left\|\partial_t^2 R_{\lambda} \right\|_{L^2(Q)} + \left\| \nabla \partial_t R_{\lambda} \right\|_{L^2(Q)} \nonumber \\ 
   & \leq C \left\| F_{\lambda} \right\|_{H^1(0,T;L^2(\Omega))},
\end{align}
which implies that $\|R_\lambda\|_{L^2(Q)}$, $\left\| \nabla R_{\lambda} \right\|_{L^2(Q)}$ and $\|\partial_t R_\lambda\|_{L^2(Q)}$ tend to zero as $\lambda\to\infty$.
\end{proposition}

\subsection{Proof of Theorem \ref{thm1}}
\begin{proof}
In the integral identity \eqref{identity}, we choose
\begin{align*}
    u_j = \mathrm{e}^{\mathrm{i} \lambda( \frac{1}{\sqrt{b}} w \cdot x +t)} (a_{j,0} + \lambda^{-1} a_{j,1}) +R_{j,\lambda},
\end{align*}
where $a_{j,0}$ solves \eqref{trans:a0} for $j=1,2$ and $w \in \mathbb S^{n-1}$. By a similar GO construction for the backward problem, we choose $u_0$ in the form
\begin{equation*}
    u_0 = \mathrm{e}^{-2\mathrm{i} \lambda( \frac{1}{\sqrt{b}} w \cdot x +t)} (a_{0,0} + \lambda^{-1} a_{0,1}) +R_{0,\lambda},
\end{equation*}
where the principal amplitude coefficient $a_{0,0}$ is given by
\begin{equation*}
    \partial_t a_{0,0} - \sqrt{b}w \cdot \nabla a_{0,0} - \frac{\gamma}{2} a_{0,0} =0,
\end{equation*}
 $\|R_{0,\lambda}\|_{L^2(Q)}$, $\|\partial_t R_{0,\lambda}\|_{L^2(Q)}$ and $\left\| \nabla R_{0,\lambda} \right\|_{L^2(Q)}$ tend to zero as $\lambda\to\infty$.

Substituting the above ansatz for $u_i, \ i=0,1,2$ into \eqref{identity} and dividing by $\lambda^3$. Letting $\lambda \to \infty$,  we obtain
\begin{equation}\label{lam-3}
    \lim\limits_{\lambda \to \infty} \lambda^{-3} \int_Q \beta \partial_tu_1 \partial_tu_2  \partial_tu_{0} + \kappa\nabla u_1 \cdot \nabla u_2 \partial_t u_0  \mathrm{d}x  \mathrm{d}t = 2\mathrm{i} \int_Q (\beta + \frac{1}{b} \kappa) a_{1,0} a_{2,0} a_{0,0} \mathrm{d}x \mathrm{d}t.
\end{equation}

Since \(T>\frac{1}{\sqrt b}\operatorname{diam}(\Omega)\), for every oriented line with direction \(w\in\mathbb S^{n-1}\) intersecting \(\Omega\), we can choose \(p\notin \overline{\Omega}\)  such that the segment
\[
\{L(t)= p-\sqrt b\,t\,w :t \in [0,T] \},
\]
traverses the intersection of the line with \(\Omega\) and $p- \sqrt b T w \notin \overline \Omega$.

Then we choose 
\begin{equation}
a_{j,0}= e^{-\frac{\gamma}{2}t}\,\phi_{\epsilon}\!\left(x+\sqrt{b}\,t\,w-p\right), \qquad j=1,2,
\end{equation}
and 
\begin{equation}
a_{0,0}= e^{\frac{\gamma}{2}t}\,\phi_{\epsilon}\!\left(x+\sqrt{b}\,t\,w-p\right).
\end{equation}
Since $p-\sqrt b T w \notin \overline \Omega$, the support of $a_{0,0}(T,\cdot)$ is disjoint from $\overline \Omega$. Hence, the amplitude function $a_{0,0}$ satisfies the terminal condition at time $T$.

Inserting the expressions of $a_{j,0}$ for $j=0,1,2$ into \eqref{lam-3}, we have
\begin{equation*}
    \int_0^T \int_{\Omega} (\beta + \frac{1}{b}\kappa) e^{-\frac{\gamma}{2}t} \phi_{\epsilon}^3(x+ \sqrt{b}tw -p) \mathrm{dx} \mathrm{d}t=0.
\end{equation*}
Since $\beta + \frac{1}{b}\kappa$ is continuous in $\Omega$, by shrinking the support of $\phi_{\epsilon}$, we have
\begin{align}\label{lineint}
     &\lim\limits_{\epsilon \to 0} \int_0^T e^{-\frac{\gamma}{2}t} \int_{\Omega} (\beta + \frac{1}{b}\kappa)(x)\phi_{\epsilon}^3(x+ \sqrt{b}tw -p) \mathrm{d}x \mathrm{d}t \\
     &= \int_0^T (\beta + \frac{1}{b}\kappa)(p- \sqrt{b}tw)e^{-\frac{\gamma}{2}t} \mathrm{d}t =0 .
\end{align}
Since $T > \frac{1}{\sqrt{b}} \operatorname{diam}(\Omega)$ and \(\gamma>0\), after extending $\beta+\frac{\kappa}{b}$ by zero outside $\Omega$, the line integral \eqref{lineint} can be viewed as an  attenuated X-ray transform. 
By the injectivity of the attenuated X-ray transform in  ${\mathbb {R}}^n$(see 
\cite[Corollary 2.5]{novikov2002inversion}), we obtain $\beta+\frac{\kappa}{b}=0$ in $\Omega$.\\
Next, by requiring the coefficient of order $\lambda^2$ to vanish, we obtain
\begin{align}
     \lim\limits_{\lambda \to \infty} &\lambda^{-2} \int_Q \beta \partial_tu_1 \partial_tu_2  \partial_tu_{0} + \kappa\nabla u_1 \cdot \nabla u_2 \partial_t u_0  \mathrm{d}x  \mathrm{d}t \notag \\
    & = \int_Q 2 \beta \partial_t (a_{1,0} a_{2,0} a_{0,0})-(3\beta +\frac{1}{b} \kappa)a_{1,0} a_{2,0} \partial_t a_{0,0} + 2\kappa\frac{1}{\sqrt{b}} w \cdot \nabla(a_{1,0} a_{2,0})a_{0,0} \mathrm{d}x\mathrm{d}t \notag \\
    & + 2i \int_Q (\beta + \frac{1}{b} \kappa) (a_{1,0} a_{2,1} a_{0,0} + a_{1,0} a_{2,0} a_{0,1} + a_{1,1} a_{2,0} a_{0,0}) \mathrm{d}x \mathrm{d}t \notag \\
    & = 2\int_Q \beta \partial_t (a_{1,0} a_{2,0}) a_{0,0} + \kappa\frac{1}{\sqrt{b}} w \cdot \nabla(a_{1,0} a_{2,0})a_{0,0} \mathrm{d}x\mathrm{d}t. \label{ide1}
\end{align} 
Set $a_{j,0}= e^{-\frac{\gamma}{2}t} \phi_{\epsilon}(x+\sqrt b t w-p)$ for $j=1,2$ and $a_{0,0}= e^{\frac{\gamma}{2}t}\phi_{\epsilon} (x+\sqrt{b}t w -p)$ where $\phi_{\epsilon}$ is as defined previously. Substituting these choices of $a_{j,0}$, $j=0,1,2$, into \eqref{ide1}, we obtain
\begin{align*}
&2\int_Q \beta \partial_t(\mathrm{e}^{-\gamma t} \phi_{\epsilon}^2)\mathrm{e}^{\frac{\gamma}{2} t} \phi_{\epsilon}+\kappa \mathrm{e}^{-\frac{\gamma}{2} t} \phi_{\epsilon} \frac{1}{\sqrt b} w \cdot \nabla \phi^2_{\epsilon} \mathrm{d}x \mathrm{d} t \\
=&2\int_{0}^T \int_{\Omega} -\gamma \mathrm{e}^{-\frac{\gamma}{2}t} \beta \phi_{\epsilon}^3+2\sqrt b (\beta+\frac{\kappa}{b}) \phi_{\epsilon}^2 \mathrm{e}^{-\frac{\gamma}{2}t} \nabla \phi_{\epsilon} \cdot w \mathrm{d}x \mathrm{d}t \\
=&-2\gamma \int_{0}^T \int_{\Omega} \mathrm{e}^{-\frac{\gamma}{2}t} \beta \phi_{\epsilon}^3(x+\sqrt b t w -p) \mathrm{d} x \mathrm{d} t.
\end{align*}
Applying the same argument as above, we obtain $\beta=0$ in $\Omega$. Thus $\kappa=0$ in $\Omega$. Therefore, the DtN map $\Lambda^\textrm{K}_{\beta,\kappa}$ uniquely determines $\beta$ and $\kappa$.  
\end{proof}

\section{Gaussian Beam Solutions}\label{Gauss}

We will focus on the  Riemannian manifold $(M,g)$ of dimension $3$.
In this section, we will construct exact solutions $u$ to the linear MGT equation with zero initial data
\begin{align}\label{MGTmanifold}
&Q u:=\partial_{t}^3 u +\alpha \partial_t^2 u - b \Delta_g \partial_t u -c^2 \Delta_g u= 0
 && \text{in }\mathcal{M},  \nonumber \\ 
&u = \partial_t u=\partial_t^2 u = 0 
    && \text{on } \{t=0\} \times M,
\end{align}
of the form
\[   u(t,x)=\mathrm{e}^{\mathrm{i} \rho \varphi(t,x)}  a_{\rho}(t,x)+\mathcal{R}_{\rho}(t,x),  \]
with a large parameter $\rho$. The principal term  
\begin{equation} \label{Gaussian}
v_{\rho}:=\mathrm{e}^{\mathrm{i} \rho \varphi(t,x)}  a_{\rho}(t,x),
\end{equation}
called the Gaussian beam solution, is concentrated near a null geodesic. 
The phase function $\varphi$ is complex-valued and $a_{\rho}$ is a smooth amplitude term.  The remainder term $\mathcal{R}_{\rho}$ tends to zero in a suitable sense as $\rho  \to \infty$. Gaussian beam solutions have also been used to study inverse problems for both elliptic and hyperbolic equations, see \cite{dos_santos_ferreira_calderon_2016,feizmohammadi_inverse_2020} and \cite{bao_sensitivity_2014,feizmohammadi_recovery_2021,feizmohammadi_recovery_2022,hintz_inverse_2022,jiang_inverse_2025,lassas_gaussian_2025,lassas_stability_2025,uhlmann_inverse_2021}.
\subsection{Fermi coordinates}
In this subsection, we recall Fermi coordinates near a null geodesic $\varsigma$, that is, a geodesic with a light-like tangent vector $\dot{\varsigma}$.   

We first extend $(M,b^{-1}g)$ to a larger nontrapping manifold $(\widetilde M,b^{-1}g)$. Consider the Lorentzian manifold $\left(\widetilde {\mathcal{M}}:= [0,T]\times \widetilde M, \bar g\right)$, where $\bar g=-dt^2+{b}^{-1}g$. We introduce Fermi coordinates in a neighborhood of a null geodesic $\varsigma$ in $\widetilde {\mathcal {M}}$. Assume that $\varsigma(t)=(t,\sigma(t))$, where $\sigma$ is a unit-speed geodesic in the Riemannian manifold $(\widetilde M,b^{-1}g)$. Then $\dot \varsigma$ is null since $\bar g(\dot \varsigma,\dot \varsigma)=-1+|\dot \sigma|_{b^{-1}g}^2=0$. Assume that $\varsigma$ passes through a point $(t_0,x_0)\in (0,T)\times M$, that is, $\sigma(t_0)=x_0 \in M$. By the nontrapping condition on $(\widetilde M,b^{-1}g)$, $\varsigma$ joins two points $(t_{-},\sigma(t_{-}))$ and $(t_{+},\sigma(t_{+}))$, where $t_{-}, t_+ \in (0,T)$ and $\sigma(t_{-}),\sigma(t_+)\in \partial M$. Extend $\varsigma$ to $\widetilde{\mathcal{M}}$ such that $\sigma(\cdot)$ is well defined on $[t_{-}-\epsilon,t_+ +\epsilon]\subset(0,T)$ with a small constant $\epsilon$. 

We will follow the construction of Fermi coordinates in \cite{feizmohammadi_recovery_2021}. Given $\sigma(t_0)=x_0 \in M$, let $T_{\sigma(t_0)} M=T_{x_0} M$ denote the tangent space at $x_0$ and $\{ \dot{\sigma}(t_0),\alpha_2,\alpha_3   \}$ form an orthonormal basis of $(T_{x_0} M,b^{-1}g)$. Let $s$ denote the arc length along $\sigma$ from $x_0$. We note that $s$ can be positive or negative and $\varsigma(t_0+s)=(t_0+s,\sigma(t_0+s))$. For $k=2,3$, let $e_k(s)\in T_{\sigma(t_0+s)}M$ be the parallel transport of $\alpha_k$ along $\sigma$ to the point $\sigma(t_0+s)$.

Define the coordinate system $(y^0=t,y^1=s,y^2,y^3)$ by the map $\mathcal{F}_1:\mathbb{R}^{1+3}\to  \widetilde {\mathcal{M}}$
\begin{equation} \label{transform1}
\mathcal{F}_1(y^0=t,y^1=s,y^2,y^3)=\left(t,\exp_{\sigma(t_0+s)}(y^2 e_2(s)+y^3 e_3(s))\right).
\end{equation}
In the new coordinates, the null geodesic $\varsigma$ is represented by $\{ (t,s,y^2,y^3):t=s,y^2=y^3=0\}$. On the null geodesic $\varsigma$, the Lorentzian metric $\bar g=-dt^2+b^{-1}g$ satisfies 
\begin{equation}
\bar g|_{\varsigma}=-dt^2+\sum_{j=1}^3 (dy^j)^2 \ \text{and} \ \ \frac{\partial \bar g_{jk}}{\partial y^i}\Bigg|_{\varsigma}=0 , \ 1\leq i,j,k \leq3.
\end{equation}
Next, introduce new coordinates $(z^0,z^{\prime}):=(z^0,z^1,z^2,z^3)=\mathcal{F}_2(y^0=t,y^1=s,y^2,y^3)$ by the linear transformation $\mathcal{F}_2$ defined by
\begin{align}
&z^0=\tau=\frac{1}{\sqrt 2}(t-t_0+s), \ \ z^1=r=\frac{1}{\sqrt 2}(-t+t_0+s),   \label{transform2} \\
&z^j=y^j,   \ \ j=2,3.  \nonumber
\end{align}
 Denote $\tau_{\pm}=\sqrt{2}(t_{\pm}-t_0)$. The neighborhood of the null geodesic $\varsigma$ is denoted by $\mathcal{V}_{\epsilon,\delta}$ as 
\begin{equation}\label{nbh}
\mathcal{V}_{\epsilon,\delta}=\left\{(\tau,z^{\prime}) \in  \widetilde{ \mathcal {M}}: \tau \in\left[\tau_{-}-\frac{\epsilon}{\sqrt 2},\tau_+ +\frac {\epsilon}{\sqrt 2} \right],|z^{\prime}|:=\sqrt{\sum_{i=1}^3(z^i)^2} <\delta \right\}.
\end{equation}
The Fermi coordinates $(z^0,z^1,z^2,z^3)=(\tau,r,z^2,z^3)$ in $\mathcal {V}_{\epsilon,\delta}$ near $\varsigma$ are given by
\[
\mathcal {F}:=\mathcal {F}_1\circ \mathcal {F}_2^{-1}:\ \mathcal{U}\subset\mathbb R^{1+3}\longrightarrow \mathcal {V}_{\epsilon,\delta}\subset \widetilde{\mathcal M},
\]
where $\mathcal{U} \subset \mathbb{R}^{1+3}$ is open.
Then on $\varsigma$ we have 
\begin{equation}\label{gjk}
\bar g|_{\varsigma}=2 d\tau dr+\sum_{j=2}^3 (dz^j)^2  \quad \text{and} \quad \frac{\partial \bar g_{jk}}{\partial z^i}\Bigg|_{\varsigma}=0 \ \ \text{for } 0\leq i,j,k\leq 3. 
\end{equation}

\subsection{Eikonal and transport equations} 
Recall $\bar g=-dt^2+b^{-1}g$. Then in local coordinates, we have
\begin{align*}
\Delta_{\bar g}u=&-\partial_t^2u+\sum_{i,j=1}^3 \left|\bar g \right|^{-\frac{1}{2}}\partial_{i} \left(|\bar g|^{\frac{1}{2} }  \bar g^{ij} \partial_j u \right) 
=-\partial_t^2u+\sum_{i,j=1}^3 b^{\frac{3}{2}}|g|^{-\frac{1}{2}}\partial_i\left( b^{-\frac{1}{2}} |g|^{\frac{1}{2}}g^{ij}\partial_ju  \right) \\
=&-\partial_t^2u+b \sum_{i,j=1}^3|g|^{-\frac{1}{2}}\partial_i\left(  |g|^{\frac{1}{2}}g^{ij}\partial_ju  \right)   -\frac{1}{2}\sum_{i,j=1}^3 g^{ij}\partial_i b \partial_ju \\
=&-\partial_t^2 u+b\Delta_g u-\sum_{i,j=1}^3 \bar g^{ij}\partial_i (\log  \sqrt b )\partial_ju=-\partial_t^2 u+b\Delta_g u-\langle \mathrm{d} \log \sqrt b , \mathrm{d} u \rangle_{\bar g}.
\end{align*}
Notice that
\begin{align*}
Q  u=&-(-\partial_t^3u+b \Delta_g\partial_t u)-\frac{c^2}{b}(-\partial_t^2 u+b\Delta_gu)+\gamma \partial_t^2u \\
=&-\Delta_{\bar g} \partial_t u- \langle \mathrm{d} \log\sqrt b,\mathrm{d}\partial_t u \rangle_{\bar g}-\frac{c^2}{b}(\Delta_{\bar g} u+\langle  \mathrm{d} \log \sqrt b, \mathrm{d}u \rangle_{\bar g} )+\gamma \partial_t^2u.
\end{align*}
Substituting WKB ansatz \eqref{Gaussian} into the operator $Q$, we obtain
\begin{equation}
Q \left(\mathrm{e}^{\mathrm{i}\rho \varphi} a\right)=\mathrm{e}^{\mathrm{i}\rho \varphi}(\rho^3 \mathcal{T}_3a+ \rho^2 \mathcal{T}_2 a+\rho \mathcal{T}_1a+Qa),
\end{equation}
where we introduce the following operators
\begin{align*}
&\mathcal{T}_3a=\mathrm{i} \partial_t \varphi a\langle  \mathrm d \varphi ,  \mathrm d \varphi \rangle_{\bar g}, \\
&\mathcal{T}_2 a= \langle \mathrm{d} \varphi , \mathrm{d}\varphi \rangle_{\bar g}  a_t   +2\langle \mathrm{d} \varphi, \mathrm{d} ( \varphi_t a)   \rangle_{\bar g} + \\
&\left(\langle \mathrm{d} \log \sqrt b,\mathrm{d} \varphi\rangle_{\bar g} \varphi_t  +\frac{c^2}{b}\langle \mathrm{d}\varphi,\mathrm{d} \varphi  \rangle_{\bar g} +\left(\Delta_{\bar g} \varphi \right) \varphi_t-\gamma (\varphi_t)^2\right) a , \\
&\mathcal{T}_1 a=\mathrm{i} (-2 \langle \mathrm{d} \varphi , \mathrm{d}a_t \rangle_{\bar g}-2\frac{c^2}{b} \langle \mathrm{d} \varphi , \mathrm{d}a \rangle_{\bar g}
+\left(2\gamma  \varphi_t-\Delta_{\bar g} \varphi-\langle \mathrm{d} \log \sqrt b ,\mathrm{d} \varphi \rangle_{\bar g}\right)a_t \\
&\qquad
+\left( \gamma \varphi_{tt}-\frac{c^2}{b}\Delta_{\bar g} \varphi-\frac{c^2}{b}\langle \mathrm{d} \log \sqrt b, \mathrm{d} \varphi  \rangle_{\bar g}  \right)a 
-\Delta_{\bar g}(\varphi_t a)- \langle \mathrm{d} \log \sqrt b , \mathrm{d} (\varphi_t a)   \rangle_{\bar g} ),
\end{align*}
We will construct an approximate Gaussian beam solution  of the form $v_{\rho}=\mathrm{e}^{\mathrm{i} \rho \varphi}a_{\rho}$ with 
\begin{equation}\label{construct}
\varphi=\sum_{j=0}^N \varphi_j(\tau,z^{\prime}), \ \ a_{\rho}(\tau,z^{\prime})=\chi\left(\frac{|z^{\prime}|}{\delta}\right)\sum_{k=0}^{N+1} \rho^{-k}a_k(\tau,z^{\prime}),  \  \ a_{k}(\tau,z^{\prime})=\sum_{j=0}^N a_{k,j}(\tau,z^{\prime}),
\end{equation}
defined in $\mathcal{V}_{\epsilon,\delta}$ introduced in \eqref{nbh}. For each $j,k$, $\varphi_j$ and $a_{k,j}$ are complex-valued homogeneous polynomials of degree $j$ with respect to the variables $z^{i}$ with $i=1,2,3$. The smooth function $\chi: \mathbb{R}\to [0,+\infty)$ satisfies $\chi(t)=1$ for $|t| \leq \frac{1}{4}$ and $\chi(t)=0$ for $|t| \ge \frac{1}{2}$.  Since 
$[t_- - \epsilon,\; t_+ + \epsilon]\subset (0,T)$, for sufficiently small \(\delta>0\), the neighborhood \(\mathcal V_{\epsilon,\delta}\) is disjoint from \(\{t=0\}\times M\) and \(\{t=T\}\times M\). After zero extension of $a_{\rho}$ in $\mathcal{V}_{\epsilon,\delta}$, we have
\[
a_\rho\big|_{t=0}=a_\rho\big|_{t=T}=0
\quad \text{in } \widetilde{\mathcal M}.
\]

We choose the eikonal equation to be
\begin{align} \label{eikonal2}
\mathcal{S}(\varphi):=\langle \mathrm{d}\varphi,\mathrm{d} \varphi  \rangle_{\bar g},
\end{align}
such that $\mathcal{S}(\varphi)$ vanishes up to order $N$ on the null geodesic $\varsigma$ with respect to the variable $z^{\prime}$. That is, in terms of Fermi coordinates $z=(z^0=\tau,z^1=r,z^2,z^3)$, we need that 
\begin{equation} \label{eikonal}
\frac{\partial^{\Theta}} {\partial z^{\Theta}}\big(\langle \mathrm{d}\varphi,\mathrm{d} \varphi  \rangle_{\bar g}\big) (\tau,0,0,0)=0 \quad \text{for } \tau\in\left[\tau_{-}-\frac{\epsilon}{\sqrt 2},\tau_+ +\frac{\epsilon}{\sqrt 2}\right],
\end{equation}
for all $\Theta=(0,\Theta_1,\Theta_2,\Theta_3)$ with $|\Theta| \leq N$. We will also require that the following transport equations hold along the null geodesic $\varsigma$ up to order $N$
\begin{align}
&\frac{\partial^{\Theta}} {\partial z^{\Theta}}\left( \mathcal{T}_2a_0  \right)(\tau,0,0,0)=0,  \quad \label{principle} \\
&\frac{\partial^{\Theta}} {\partial z^{\Theta}}\left( \mathcal{T}_1a_0+\mathcal{T}_2a_1  \right)(\tau,0,0,0)=0, \label{2trans}\\
&\frac{\partial^{\Theta}} {\partial z^{\Theta}}\left(Qa_i+\mathcal{T}_1a_{i+1}+ \mathcal{T}_2a_{i+2}  \right)(\tau,0,0,0)=0  , \ \text{for } i=0,1,\cdots,N-1, \label{alltrans}
\end{align}
for $\Theta=(0,\Theta_1,\Theta_2,\Theta_3)$ with $|\Theta| \leq N$. 

\subsection{Construction of the phase function and the amplitude function}\label{4.3}
We begin by solving the eikonal equation \eqref{eikonal}. We refer the reader to \cite{feizmohammadi_recovery_2021,feizmohammadi_recovery_2022} for more details on the construction of the phase function. For $|\Theta|=0$, we obtain the equation
\[
\Bigl(\bar g^{kl}\,\partial_k\varphi\,\partial_l\varphi\Bigr)\Big|_{\varsigma}=0.
\]
Recalling that $\bar g|_{\varsigma}=2dz^0dz^1+(dz^2)^2+(dz^3)^2$, this reduces to
\begin{equation}\label{eq:eikonal-n3}
2\,\partial_0\varphi\,\partial_1\varphi + (\partial_2\varphi)^2+(\partial_3\varphi)^2=0.
\end{equation}
Similarly, for $|\Theta|=1$, we obtain (using $\partial_i \bar g^{jk}|_{\varsigma}=0$)
\begin{equation}\label{eq:m1-n3}
\Bigl(\bar g^{kl}\,\partial_{k\alpha}^2\varphi\,\partial_l\varphi\Bigr)\Big|_{\varsigma}=0,
\qquad \alpha=1,2,3.
\end{equation}

Recalling the definition of the phase $\varphi$ from equation \eqref{construct}, we set $\varphi_0=0$ and $\varphi_1=r=\frac{-t+t_0+s}{\sqrt 2}$ such that \eqref{eq:eikonal-n3} and \eqref{eq:m1-n3} are satisfied.
Next, we choose $\varphi_2$ in the form
\begin{equation}\label{phase}
\varphi_2(\tau,z')= \sum_{1\leq i,j\leq3}H_{ij}(\tau)z^iz^j,
\end{equation}
where $H(\tau)\in\mathbb{C}^{3\times 3}$ is symmetric and $\Im H(\tau)>0 \text{ for } \tau\in(\tau_{-}-\frac{\epsilon}{2},\tau_+ +\frac{\epsilon}{2})$.

Next we consider the case $|\Theta|=2$, we require that
\[
\frac{\partial^2}{\partial z_i\partial z_j}
\Bigl(\bar g^{kl}\,\partial_k\varphi\,\partial_l\varphi\Bigr)\Big|_{\varsigma}=0,
\qquad \text{for all } 1\leq i,j\leq3.
\]
This simplifies to
\begin{equation}\label{eq:pre-riccati-n3}
\left(
\partial_{ij}^2\bar g^{11}
+2\bar g^{10}\,\partial_{0ij}^3\varphi
+2\sum_{k=2}^{3}\partial_{ki}^2\varphi\,\partial_{kj}^2\varphi
\right)\Bigg|_{\varsigma}=0,\qquad i,j=1,2,3.
\end{equation}
Consequently, $H$ satisfies the Riccati equation
\begin{equation}\label{eq:riccati-n3}
\frac{d}{d\tau}H + HCH + D = 0 \ \ \ \tau\in(\tau_{-}-\frac{\epsilon}{2},\tau_+ +\frac{\epsilon}{2}),\quad H(0)=H_0 \ \text{with }\Im H_0>0,
\end{equation}
where
\[
C=\begin{pmatrix}
0&0&0\\[2pt]
0&2&0\\[2pt]
0&0&2
\end{pmatrix},
\qquad
D(\tau):=\frac14\Bigl(\partial_{ij}^2\bar g^{11}\big|_{\varsigma}\Bigr)_{1\le i,j\le 3}.
\]
We recall the following result from \cite{feizmohammadi_recovery_2021} on the solvability of the Riccati equation \eqref{eq:riccati-n3}.
\begin{lemma}
The Riccati equation \eqref{eq:riccati-n3} has a unique solution, and the solution $H$ is symmetric and $\Im (H(\tau))>0$ for $\tau \in \left( \tau_{-}-\frac{\epsilon}{2},\tau_{+}+\frac{\epsilon}{2}    \right)$. 
Moreover, for solving the above Riccati equation, one can express the solution $H$ as
\begin{equation*}
H(\tau)=Z(\tau)Y(\tau)^{-1},
\end{equation*}
where $Y(\tau),Z(\tau)\in\mathbb{C}^{3\times 3}$ solve the ODEs
\begin{align}
\frac{d}{d\tau}Y(\tau) &= CZ(\tau), \qquad Y(0)=Y_0, \label{eq:Ysys-n3}\\
\frac{d}{d\tau}Z(\tau) &= -D(\tau)Y(\tau),\qquad Z(0)=Z_0=H_0Y_0. \label{eq:Zsys-n3}
\end{align}
In addition, $Y(\tau)$ is non-degenerate for $\tau\in \left( \tau_{-}-\frac{\epsilon}{2},\tau_{+}+\frac{\epsilon}{2}    \right) $ and 
\begin{equation}\label{station}
\det \Im(H(\tau))|\det Y(\tau)|^2=c_0,
\end{equation}
where $c_0$ is a constant independent of $\tau$.
\end{lemma}
Combining \eqref{eq:Ysys-n3} and \eqref{eq:Zsys-n3}, we see that the matrix $Y(\tau)$ satisfies 
\begin{equation}\label{Y}
\frac{\mathrm{d}^2}{\mathrm{d}\tau^2}Y(\tau)+CD(\tau)Y(\tau)=0, \ \ Y(0)=Y_0, \ \ \frac{\mathrm{d}Y}{\mathrm{d}\tau}\left(0\right)=CZ_0.
\end{equation}
Furthermore, we obtain the function $\varphi_j$ by recursively solving the linear ODEs obtained from the Taylor coefficients of $\mathcal{S}(\varphi)$,  $j \ge 3$, see \cite[Subsection 4.2.1]{feizmohammadi_recovery_2022}. Then we have $\mathcal{S}(\varphi)=\mathcal{O}(|z^{\prime}|^{N+1})$.


Next, we construct the amplitude function $a_{\rho}$. We note that 
\begin{equation}
\left( \Delta_{\bar g}\varphi \right)\big|_{\varsigma}=\sum_{i,j=0}^3 {\bar g}^{ij} \partial_{ij}^2 \varphi\big|_{\varsigma}=\sum_{j=2}^3\partial_{jj} \varphi\big|_{\varsigma}=\mathrm{Tr}(CH),
\end{equation}
and along the null geodesic $\varsigma$, $\mathrm{Tr}(CH)$ can be expressed as
\begin{align*}
\mathrm{Tr}(CH(\tau))=\mathrm{Tr}(CZ(\tau)Y(\tau)^{-1})=\mathrm{Tr}\big(\frac{\mathrm{d}}{\mathrm{{d}\tau}}Y(\tau) Y(\tau)^{-1}\big)=\frac{\mathrm{d}}{\mathrm{d}\tau}(\log(\det Y(\tau))).
\end{align*}
Then for $|\Theta|=0$ the equation \eqref{principle} can be rewritten as 
\begin{equation}\label{1principle}
\Big(\varphi_t\big(2\langle \mathrm{d} \varphi,\mathrm{d}a_0\rangle_{\bar g}+(\langle  \mathrm{d} \log \sqrt{b},\mathrm{d}\varphi\rangle_{\bar g}+\mathrm{Tr}(CH)-\gamma \varphi_t)a_0     \big)  \Big)\bigg|_{\varsigma}=0.
\end{equation}
Combining \eqref{transform2} with \eqref{phase}, we have 
\begin{equation*}
\sqrt{2}\partial_t \varphi|_{\varsigma}=\bigg(\partial_{\tau}\Big(r+\sum_{i,j=1}^3H_{i,j}(\tau)z^{i}z^{j} \Big)-\partial_{r}\Big(r+\sum_{i,j=1}^3H_{i,j}(\tau)z^{i}z^{j} \Big)\bigg)\Bigg|_{\varsigma}=-1,
\end{equation*}
which implies that equation \eqref{1principle} simplifies to
\begin{equation}\label{2principle}
2\partial_{\tau}a_{0,0}+\left(\partial_{\tau} \log\sqrt{b} +\partial_{\tau}(\log(\det(Y(\tau))))+\frac{\sqrt 2}{2} \gamma \right)a_{0,0}=0   \ \  \text{in } \varsigma   .
\end{equation}
We set the solution to \eqref{2principle} explicitly as
\begin{equation}
a_{0,0}|_{\varsigma}=\left(\sqrt{b(\tau)} \det Y(\tau)  \right)^{-\frac12} \mathrm{e}^{-\frac{\sqrt 2}{4}\int_{\tau_{-}}^{\tau}\gamma(s,0,0,0)\mathrm{d}s}.
\end{equation}
We note that $a_{\rho}$ is concentrated in the null geodesic $\varsigma$ and along $\varsigma$
\begin{equation*}
a_{\rho}(\varsigma(\tau))=a_{0,0}(\varsigma(\tau))+\mathcal{O}(\rho^{-1}), \qquad
\tau\in \left(\tau_- - \frac{\epsilon}{2},\,\tau_+ + \frac{\epsilon}{2}\right).
\end{equation*}
The subsequent terms $a_{0,j}$ with $j=1,\cdots,N$ can be constructed by solving linear first order ODEs. Taking $|\Theta|=j$ in equation \eqref{principle}, we obtain the following first order ODE for the homogeneous polynomial $a_{0,j}(\tau,z^{\prime})$ 
\begin{equation}
-\frac{\sqrt 2}{2}\left(2\partial_{\tau}a_{0,j}+\left(\partial_{\tau} \log(\sqrt{b}\det Y(\tau)) +\frac{\sqrt 2}{2} \gamma \right)a_{0,j} \right)+ \mathcal{E}_j=0,
\end{equation}
with $\mathcal{E}_j$ is a homogeneous polynomial of degree $j$ in the $z^{\prime}$ coordinates. It has a unique solution prescribing arbitrary initial data at the point $\tau=0$. Therefore, we obtain the leading amplitude function $a_0$.

Letting $|\Theta|=0$ and using the fact $\partial_t \varphi =-\frac{\sqrt 2}{2}$ along $\varsigma$, equations \eqref{2trans} and \eqref{alltrans} simplify to 
\begin{align}
&-\frac{\sqrt 2}{2}\left(2\partial_{\tau}a_{1,0}+\left(\partial_{\tau} \log(\sqrt{b}\det Y(\tau)) +\frac{\sqrt 2}{2} \gamma \right)a_{1,0} \right)=-\mathcal{T}_1a_{0,0}  \ \ \text{in } \varsigma,\\
&-\frac{\sqrt 2}{2}\left(2\partial_{\tau}a_{i+2,0}+\left(\partial_{\tau} \log(\sqrt{b}\det Y(\tau)) +\frac{\sqrt 2}{2} \gamma \right)a_{i+2,0} \right)=-\left(\mathcal{T}_1a_{i+1,0}+Qa_{i,0} \right)  \ \ \text{in } \varsigma,
\end{align}
$\text{for } i=0,1,\cdots,N-1$. Consequently, $a_{1,0}$ is determined by solving a linear first-order ODE along $\varsigma$ with the source term $-\mathcal{T}_1 a_{0,0}$ and, by reduction, $a_{k+2,0}$ is obtained as the source term $-\left(\mathcal{T}_1a_{k+1,0}+Qa_{k,0} \right)$. This provides an inductive construction of $a_{k,0}|_{\varsigma}$ for $k=1,2,\cdots,N+1$ and we omit the explicit solution.  

To determine the subsequent terms $a_{k,j}$, we need to solve equations \eqref{2trans} and \eqref{alltrans} with $|\Theta|=j$. This can be solved by the above argument for $a_{0,j}$, and we omit the details. 

\begin{lemma}
Let $v_{\rho}=\mathrm{e}^{\mathrm{i}\rho \varphi} a_{\rho}$ be an approximate Gaussian beam solution of order $N$ in Subsection \ref{4.3}. For $\rho \gg 1$, we have 
\begin{align}
&\|Q v_{\rho}\|_{H^{k}(\mathcal M)} \lesssim \rho^{-K},  \label{remainder}
\end{align}
where $K=\frac{N}{2}-k-\frac{7}{4}$.
\end{lemma}
\begin{proof}
Our proof follows the argument of \cite[Proposition 10]{lassas_stability_2025}. By redefining $\delta$ small enough, we have 
\begin{equation*}
|\mathrm{e}^{\mathrm{i} \rho \varphi}| \leq \mathrm{e}^{-c \rho |z^{\prime}|^2}.
\end{equation*}
Taking $k$ derivatives of $Q(\mathrm{e}^{\mathrm{i} \rho \varphi}a_{\rho})$ gives
\begin{align} \label{kderi}
\big|\nabla^k Q(\mathrm{e}^{\mathrm{i} \rho \varphi}a_{\rho})\big| \leq C_0 \mathrm{e}^{-c\rho |z^{\prime}|^2} \sum_{l=0}^k \rho^{k-l} \left( \rho^{3}|z^{\prime}|^{N+1-l} +\rho^2 |z^{\prime}|^{N+1-l}+\rho^{-N}  \right).
\end{align}
We calculate that the integral of \eqref{kderi} squared using polar coordinates for the $z^{\prime}-$variable and standard formula $\int_0^{\infty}  r^{l} \mathrm{e}^{-c\rho r^2} \mathrm{d}r \sim \rho^{-\frac{l+1}{2}}$ for $l \ge 0$. Then we obtain that 
\begin{align*}
\|  Q(\mathrm{e}^{\mathrm{i} \rho \varphi} a_{\rho})  \|_{H^k(\mathcal M)}^2 &\lesssim \sum_{l=0}^k \rho^{2(k-l)} \left( \int_0^{r_0} \mathrm{e}^{-2c\rho r^2} r^{n-1}(\rho^6  r^{2N+2-2l}+\rho^{-2N})\mathrm{d}r        \right) \\
&\lesssim \sum_{l=0}^k \rho^{2(k-l)}\left(\rho^6 \cdot \rho^{-\frac{n+2N+2-2l}{2}}+\rho^{\frac{-2N-n}{2}}\right) \\
&\lesssim \rho^{2k+\frac{7}{2}-N},
\end{align*}
for $\rho$ large enough and $n=3$. For $N$ sufficiently large, we set $K=\frac{N}{2}-k-\frac{7}{4}$, then we obtain that 
\begin{equation} \label{Hk}
\| Q \left(\mathrm{e}^{\mathrm{i} \rho \varphi} a_{\rho} \right)\|_{H^k(\mathcal M)} \lesssim \rho^{-K}.
\end{equation}
\end{proof}

\subsection{Estimate of the Remainder term}
After choosing the phase function $\varphi$ and the amplitude function $a_{\rho}$, we construct the remainder term $\mathcal{R}_{\rho}$.  By Lemma \ref{2},
the initial boundary value problem
\begin{equation} \label{eq:remainder2}
\left\{
\begin{aligned}
&\partial_{t}^3 \mathcal{R}_{\rho} +\alpha \partial_t^2 \mathcal{R}_{\rho} - b \Delta_g \partial_t \mathcal{R}_{\rho} -c^2 \Delta_g \mathcal{R}_{\rho}= -  Q v_{\rho}
 && \text{in } \mathcal{M},\\ 
&\mathcal{R}_{\rho} = 0 
    && \text{on } \Gamma, \\
&\mathcal{R}_{\rho} = \partial_t \mathcal{R}_{\rho}=\partial_t^2 \mathcal{R}_{\rho} = 0 
    && \text{on } \{t=0\} \times M
\end{aligned}
\right.
\end{equation} 
admits a unique solution $\mathcal{R}_{\rho}\in E^{s+2}$ such that 
\begin{equation*}
\| \mathcal{R}_{\rho}\|_{E^{s+2}} \leq C \| Q v_{\rho}  \|_{E^s}.
\end{equation*}
An application of \cite[Corollary 11]{lassas_stability_2025} and Lemma \ref{2} yields that 
\begin{equation*}
\| \mathcal{R}_{\rho}\|_{H^l(\mathcal M)}  
\lesssim \| Q(\mathrm{e}^{\mathrm{i}\rho \varphi} a_{\rho})\|_{H^k(\mathcal M)}\lesssim \rho^{-K},
\end{equation*}
where $l \in \mathbb N$ is such that $l<k$. Here $k$ and $K$ are the same as in the estimate \eqref{Hk}. Furthermore, we apply the Sobolev embedding theorem and choose $l>5$, $k>l$ and $N\ge 2k+\frac{19}{2}$  to obtain the estimate  
\begin{equation} \label{C2}
\| \mathcal{R}_{\rho}\|_{C^2(\mathcal M)} \lesssim \rho^{-3}.
\end{equation}

\begin{remark}
We remark that by the same argument we can construct the approximate Gaussian beam solutions of the form $u_{\rho}=\mathrm{e}^{\mathrm{i}\rho\varphi}a_{\rho}$ concentrated in the null geodesic $\varsigma$ for the backward equation
\begin{align}\label{MGTbackward}
&Q^{*} u:=-\partial_{t}^3 u +\alpha \partial_t^2 u +  \Delta_g (b\partial_t u) - \Delta_g (c^2u)= 0
 && \text{in }\mathcal{M},  \nonumber \\ 
&u = \partial_t u=\partial_t^2 u = 0 
    && \text{on } \{t=T\} \times M.
\end{align}
Here $\varphi$ satisfies the eikonal equation \eqref{eikonal2}, and the construction of the phase function $\varphi$ is similar. The principal amplitude function
\begin{equation}
a_{0,0}|_{\varsigma}=\left(b(\tau)^{\frac{5}{2}} \det(Y(\tau))  \right)^{-\frac{1}{2}} \mathrm{e}^{\frac{\sqrt 2}{4} \int_{\tau_{-}}^{\tau} \gamma(s,0,0,0)ds}. 
\end{equation}
Finally, the remainder estimate for $R_{\rho}=u-u_{\rho}$  holds
\begin{align} \label{back}
&\|Q^{*} u_{\rho}\|_{H^{k}(\mathcal M )} \lesssim \rho^{-K} , \quad\| R_{\rho} \|_{C^2(\mathcal M)}\lesssim \rho^{-3},
\end{align}
for $N$ and $\rho$ large enough.
We note that if $u_\rho=\mathrm{e}^{\mathrm{i}\rho\varphi}a_{\rho}$ is an approximate Gaussian beam solution of order $N$ of the backward MGT equation, then 
\begin{equation*} 
\tilde u_{\rho}=\mathrm{e}^{-\mathrm{i} \rho \bar \varphi} \bar{a}_{\rho}
\end{equation*}
is also an approximate Gaussian beam solution of order $N$ of the backward MGT equation and  satisfies the above remainder estimate \eqref{back}. Here the notation $\overline{\cdot}$ means complex conjugation.
\end{remark}

\section{The Geometric Case}\label{geometric}

In this section, we mainly provide the proofs of Theorem \ref{thm2} and \ref{thm3}. We construct special Gaussian beam solutions and insert them into the corresponding integral identity similar to \eqref{identity}; then we reduce the problem to the injectivity of a certain geodesic ray transform to show the uniqueness of nonlinear coefficients from the DtN map.
\subsection{Proof of  Theorem \ref{thm2}} \label{thm2p}   We now proceed with the proof of Theorem \ref{thm2}.
\begin{proof}
Assume that $u$ solves the equation \ref{JMGT2} with Dirichlet boundary value
\begin{equation*}
f=\epsilon_1f_1+\epsilon_2f_2,
\end{equation*}
where $f_1$ and $f_2$ vanish near $\{t=0\}$ and $\epsilon_1,\epsilon_2$ are small parameters. Let $u_j$, $j=1,2$, denote the solutions to the linearized MGT equation with boundary value $f_j$, that is
\begin{equation}\label{linearized1}
\left\{
\begin{aligned}
&\partial_t^3u_j + \alpha \partial_{t}^2 u_j  - b\Delta_g \partial_t u_j-c^2\Delta_g u_j =0
    && \text{in } \mathcal{M}, \\
&u_j= f_j
    && \text{on } \Gamma, \\
&u_j(0,x)= \partial_t u_j(0,x)=\partial_t^2u_j(0,x)=0
    && \text{on } M.
\end{aligned}
\right.
\end{equation}
Applying $\frac{\partial^2}{\partial \epsilon_1 \partial \epsilon_2}$ to \eqref{JMGT2} and setting $\epsilon_1=\epsilon_2=0$, we obtain that $w=\frac{\partial^2 u}{\partial \epsilon_1 \epsilon_2}\Big|_{\epsilon_1=\epsilon_2=0}$ satisfies the equation
\begin{equation}
\left\{
\begin{aligned}
&\partial_t^3w + \alpha \partial_{t}^2 w  - b\Delta_g \partial_t w-c^2\Delta_g w = 2\partial_t^2 \left(\beta u_1u_2 \right)
    && \text{in } \mathcal{M}, \\
&w= 0
    && \text{on } \Gamma, \\
&w(0,x)= \partial_t w(0,x) = \partial_t^2 w(0,x)=0
    && \text{on } M.
\end{aligned}
\right.
\end{equation}
We note that
\begin{equation*}
\frac{\partial^2}{\partial \epsilon_1 \partial \epsilon_2} \Lambda^\textrm{W}_{\beta}(\epsilon_1 f_1+\epsilon_2 f_2)\Big|_{\epsilon_1=\epsilon_2=0}= b\partial_{\nu}\partial_t w+c^2\partial_{\nu}w.
\end{equation*}
Assume that $u_0$ solves the backward MGT equation
\begin{equation}\label{backward5.1}
\left\{
\begin{aligned}
&-\partial_t^3u_0 + \alpha \partial_{t}^2 u_0+ \Delta_g (b\partial_t u_0)-\Delta_g (c^2u_0 )=0
    && \text{in } \mathcal{M}, \\
&u_0= f_0
    && \text{on } \Gamma, \\
&u_0(T,x)= \partial_t u_0(T,x)=\partial_t^2u_0(T,x)=0
    && \text{on } M.
\end{aligned}
\right.
\end{equation}
Using integration by parts, we obtain that
\begin{align}\label{integration}
&\int_0^T\int_{\partial  M} \frac{\partial^2}{\partial \epsilon_1 \partial \epsilon_2} \Lambda^\textrm{W}_{\beta}(\epsilon_1 f_1 +\epsilon_2 f_2)\Big|_{\epsilon_1=\epsilon_2=0}  f_0  dS_g  dt \nonumber \\
&=\int_0^T\int_{\partial M} (b\partial_{\nu}\partial_t w+c^2\partial_{\nu} w) u_0  dS_g  dt \nonumber \\
&=\int_0^T\int_{M}  (b\Delta_g \partial_t w+c^2\Delta_gw) u_0 + \langle \nabla_g \partial_tw, \nabla_g(bu_0)\rangle_g+\langle \nabla_g w,\nabla_g(c^2u_0)\rangle_g    dV_gdt \nonumber \\
&= \int_0^T\int_M (\partial_t^3w+\alpha\partial_t^2 w-2\beta \partial_t^2(u_1u_2)) u_0 +w(\Delta_g(b\partial_t u_0)-\Delta_g(c^2 u_0))    dV_gdt   \nonumber \\
&=\int_0^T\int_M -2\beta \partial_t^2(u_1u_2) u_0 +w(-\partial_t^3u_0+\alpha \partial_t^2 u_0+ \Delta_g(b\partial_t u_0)-\Delta_g(c^2 u_0))    dV_gdt   \nonumber \\
&= \int_0^T\int_M 2\beta   \partial_tu_{0} \partial_t(u_1u_2)  dV_g  dt .
\end{align}
We only need to recover the parameter $\beta(x)$ from the integral identity \eqref{integration}. Consider \eqref{integration} with $\beta$ replaced by $\beta_1$ and $\beta_2$, and subtract the two identities. Then we have
\begin{equation}\label{identity5.1}
 \int_0^T\int_M\beta \partial_t u_0 \partial_t(u_1u_2)  dV_g  dt=0,
\end{equation}
where $\beta$ denotes $\beta_1-\beta_2$.

We construct exact solutions of the form 
\begin{align*}
&u_{1}=u_{2}=\mathrm{e}^{\mathrm{i}\rho \varphi}a_{\rho}+\mathcal{R}_{\rho}=\mathrm{e}^{\mathrm{i}\rho \varphi} \chi_{\delta} \sum_{k=0}^{N+1} \rho ^{-k} a_k+\mathcal R_{\rho}, \\
&u_{0}=\mathrm{e}^{-2\mathrm{i}\rho \bar \varphi} \bar{a}_{2\rho}^{(0)}+R_{2\rho}=\mathrm{e}^{-2\mathrm{i}\rho \bar \varphi} \chi_{\delta} \sum_{k=0}^{N+1} (2\rho)^{-k} \bar a_{k}^{(0)}+R_{2\rho},
\end{align*}
where $\chi_{\delta}$ denotes $\chi\left(\frac{|z^{\prime}|}{\delta}\right)$, the corresponding Gaussian beams  concentrate near the same null geodesic $\varsigma$ and the remainder estimates 
\begin{equation*}
\|\mathcal{R}_{\rho}\|_{C^2(\mathcal M)}\lesssim \rho^{-3}, \quad \|R_{2\rho}\|_{C^2(\mathcal M)}\lesssim \rho^{-3}, 
\end{equation*}
hold.
Substituting  the above representations of $u_0,u_1,u_2$ into the integral identity 
\eqref{identity5.1}, we obtain 
\begin{align*}
& \rho^{-\frac{1}{2}}\int_0^T\int_M\beta \partial_t u_0 \partial_t(u_1u_2)  dV_g  dt\\
&=\rho^{-\frac{1}{2}}\int_0^T \int_M \beta \mathrm{e}^{2\mathrm{i}\rho \varphi}(2\mathrm{i}\rho \partial_t \varphi)\chi_{\delta}^3 a_{0}a_{0}\bar a_0^{(0)} \mathrm{e}^{-2\mathrm{i} \rho \bar \varphi}(-2\mathrm{i} \rho\partial_t \bar \varphi)dV_g dt +\mathcal{O}(\rho^{-1})             \\  
 &=4\rho^{\frac{3}{2}}\int_0^T\int_M\beta \mathrm{e}^{-4\rho \Im \varphi} \chi_{\delta}^3 a_{0}a_{0}\bar a_{0}^{(0)} |\partial_t \varphi|^2dV_g dt+\mathcal{O}(\rho^{-1}).
\end{align*}
Notice that 
\begin{equation*}
|\partial_t \varphi|^2 a_{0}a_{0}\bar a_{0}^{(0)}|_{\varsigma}=c_0 b^{-\frac{7}{4}}\mathrm{e}^{-\frac{\sqrt{2}}{4}\int_{\tau_{-}}^\tau \gamma(s,0,0,0)ds}|\det Y(\tau)|^{-1} \left(\det Y(\tau) \right)^{-\frac{1}{2}},
\end{equation*}
where $c_0$ is a constant. Using equation \eqref{gjk}, we have 
\begin{equation*}
\Big||\det \bar g(\tau,z^{\prime})|-1\Big| \leq C |z^{\prime}|^2   \quad \text{near } \varsigma . 
\end{equation*}
Using the fact that 
\begin{equation*}
dt\wedge dV_g=b^{\frac{3}{2}} dV_{\bar g}= \sqrt{|\det \bar g|(\tau,z^{\prime})}  d \tau \wedge d z^{\prime},
\end{equation*}
we have
\begin{equation*}
\lim_{\rho \to +\infty} \rho^{\frac{3}{2}}\int_{\tau_{-}-\frac{\epsilon}{\sqrt 2}}^{\tau_+ +\frac{\epsilon}{\sqrt 2}} \int_{|z^{\prime}|<\delta} \beta b^{-\frac{7}{4}} \mathrm{e}^{-\frac{\sqrt 2}{4} \int_{\tau_{-}}^{\tau}\gamma(s,0,0,0)ds}\mathrm{e}^{-4\rho \Im \varphi} |\det Y(\tau)|^{-1} (\det Y(\tau))^{-\frac{1}{2}} b^{\frac{3}{2}} d\tau \wedge d z^{\prime}=0,
\end{equation*}
with $\delta$ sufficiently small. Using the method of stationary phase together with \eqref{station}, we obtain that
\begin{align*}
&\rho^{\frac{3}{2}} \int_{|z^{\prime}|<\delta} \tilde \beta \mathrm{e}^{-4\rho \Im \varphi} dz^{\prime}=
\tilde \beta(\tau,0)\frac{C_0}{\sqrt{\det(\Im H(\tau))}}+\mathcal{O}(\rho^{-1}) \\
=&C_1\tilde \beta(\tau,0)| \det Y(\tau)|+\mathcal O(\rho^{-1}),
\end{align*}
where $\tilde \beta(\tau,0)= \beta(\tau,0) b^{-\frac{1}{4}}(\tau,0) \mathrm{e}^{-\frac{\sqrt 2}{4}\int_{\tau_{-}}^\tau \gamma(s,0,0,0)ds}$. Since $T>\mathrm{diam}_{b^{-1}g} M$, we can extract the line integral 
\begin{equation}\label{jacobi}
\int_{\sigma}   \beta(\tau,0) b^{-\frac{1}{4}}(\tau,0) \mathrm{e}^{-\frac{\sqrt 2}{4}\int_{\tau_{-}}^\tau \gamma(s,0,0,0)ds}  (\det Y(\tau))^{-\frac{1}{2}}d \tau=0 ,
\end{equation}
from the integral identity \eqref{identity5.1}. To prove the unique recovery of $\beta$ from $\Lambda^\textrm{W}_{\beta}$, it remains to invoke the injectivity of the above weighted ray transform in \eqref{jacobi}.

If $(M,b^{-1}g)$ satisfies the foliation condition, we can also use the invertibility of the weighted ray transform with a single weight established in \cite{paternain_geodesic_2019}. Then we complete the proof.
\end{proof}

\subsection{Proof of Theorem \ref{thm3}} 
First, we apply  second order linearization method to derive an integral identity.
We then insert  Gaussian beam solutions into this identity and analyze its asymptotic expansion. 
The leading-order term yields the determination of $\beta+\frac{\kappa}{b}$.
After using this relation and boundary condition $\beta_1=\beta_2$ on $\partial M$, the next term gives a weighted ray transform of $\partial_{\tau}\beta$. 
By the injectivity of the corresponding ray transforms, we obtain the unique recovery of $\beta$ and $\kappa$.
\begin{proof}
Similar to the second order linearization method used in Subsection \ref{thm2p} , assume that $u$ solves the equation \eqref{JMGT} with Dirichlet boundary value
\begin{equation*}
f=\epsilon_1f_1+\epsilon_2f_2,
\end{equation*}
where $f_1,f_2$ vanish near $\{t=0\}$ and $\epsilon_1,\epsilon_2$ are small parameters. Denote by $u_j$, $j=1,2$, the solution to the linearized MGT equation \eqref{linearized1} with boundary value $f_j$. Also assume that $u_0$ solves the backward MGT equation \eqref{backward5.1}. Applying $\frac{\partial^2}{\partial \epsilon_1 \partial \epsilon_2}$ to \eqref{JMGT} and setting $\epsilon_1=\epsilon_2=0$, we obtain that $w=\frac{\partial^2 u}{\partial \epsilon_1 \epsilon_2}\Big|_{\epsilon_1=\epsilon_2=0}$ satisfies the equation
\begin{equation}
\left\{
\begin{aligned}
&\partial_t^3w + \alpha \partial_{t}^2 w  - b\Delta_g \partial_t w-c^2\Delta_g w = 2\partial_t \left(\beta \partial_tu_1\partial_tu_2+\kappa\langle\nabla_g u_1,\nabla_gu_2\rangle_g \right)
    && \text{in } \mathcal{M}, \\
&w= 0
    && \text{on } \Gamma, \\
&w(0,x)= \partial_t w(0,x) = \partial_t^2 w(0,x)=0
    && \text{on } M.
\end{aligned}
\right.
\end{equation}
Similar to \eqref{integration}, integration by parts yields that
\begin{align} \label{integration5.2}
&\int_0^T\int_{\partial M} \frac{\partial^2}{\partial \epsilon_1 \partial \epsilon_2}\Lambda^\textrm{K}_{\beta,\kappa}(\epsilon_1 f_1+\epsilon_2f_2)\Big|_{\epsilon_1=\epsilon_2=0} u_0 dS_g dt \nonumber \\
&=\int_0^T \int_M 2\beta\partial_tu_0\partial_tu_1 \partial_tu_2+2\kappa \partial_t u_0 \langle \nabla_gu_1,\nabla_gu_2\rangle_g dV_g dt.
\end{align}
Consider \eqref{integration5.2} with $\beta$ and $\kappa$ replaced by $\beta_j$ and $\kappa_j$, and subtract the two identities. Then we have
\begin{equation}\label{identity5.2}
 \int_0^T\int_M\beta \partial_t u_0 \partial_tu_1\partial_t u_2+\kappa\partial_t u_0\langle\nabla_gu_1, \nabla_gu_2\rangle_g  dV_g  dt=0,
\end{equation}
where $\beta$ denotes $\beta_1-\beta_2$ and $\kappa$ denotes $\kappa_1-\kappa_2$.

We construct exact solutions of the form 
\begin{align*}
&u_{1}=u_{2}=\mathrm{e}^{\mathrm{i}\rho \varphi}a_{\rho}+\mathcal{R}_{\rho}=\mathrm{e}^{\mathrm{i}\rho \varphi} \chi_{\delta} \sum_{k=0}^{N+1} \rho ^{-k} a_k+\mathcal R_{\rho}, \\
&u_{0}=\mathrm{e}^{-2\mathrm{i}\rho \bar \varphi} \bar{a}_{2\rho}^{(0)}+R_{2\rho}=\mathrm{e}^{-2\mathrm{i}\rho \bar \varphi} \chi_{\delta} \sum_{k=0}^{N+1} (2\rho)^{-k} \bar a_{k}^{(0)}+R_{2\rho},
\end{align*}
where the corresponding Gaussian beams  concentrate near the same null geodesic $\varsigma$ and the remainder estimates 
\begin{equation*}
\|\mathcal{R}_{\rho}\|_{C^2(\mathcal M)}\lesssim \rho^{-3}, \quad \|R_{2\rho}\|_{C^2(\mathcal M)}\lesssim \rho^{-3}, 
\end{equation*}
hold.
Substituting  the above representations of $u_0,u_1,u_2$ into the integral identity 
\eqref{identity5.2}, we obtain  
\begin{align*}
& \rho^{-\frac{3}{2}}\int_0^T\int_M\beta \partial_t u_0 \partial_tu_1\partial_tu_2+\kappa\partial_t u_0 \langle \nabla_g u_1,\nabla_g u_2 \rangle_g  dV_g  dt\\
&=\rho^{-\frac{3}{2}}\int_0^T \int_M \beta \mathrm{e}^{2\mathrm{i}\rho \varphi}(\mathrm{i}\rho \partial_t \varphi)^2\chi_{\delta}^3 a_{0}a_{0} \bar a_0^{(0)} \mathrm{e}^{-2\mathrm{i} \rho \bar \varphi}(-2\mathrm{i} \rho\partial_t \bar \varphi)\\
&+\kappa \mathrm{e}^{-2\mathrm{i}\rho \bar \varphi} (-2\mathrm{i} \rho\partial_t \bar \varphi) \mathrm{e}^{2\mathrm{i}\rho \varphi} (\mathrm{i} \rho)^2 a_{0}a_{0} \bar a_0^{(0)} \chi_{\delta}^3 \langle d\varphi,d\varphi \rangle_g dV_g dt +\mathcal{O}(\rho^{-1})             \\ 
&=2\mathrm{i}\rho^{\frac{3}{2}}\int_{\tau_{-}-\frac{\epsilon}{\sqrt 2}}^{\tau_+ +\frac{\epsilon}{\sqrt 2}} \int_{|z^{\prime}| <\delta} \mathrm{e}^{-4\rho \Im \varphi}a_{0} a_{0} \bar a_0^{(0)} \left( \beta |\partial_t \varphi|^2 \partial_t \varphi+ \kappa \partial_t \bar\varphi \langle d\varphi, d\varphi  \rangle_g \right) b^{\frac{3}{2}} d\tau \wedge d z^{\prime} +\mathcal{O}(\rho^{-1})      \\
 &=2\mathrm{i}\rho^{\frac{3}{2}}\int_{\tau_{-}-\frac{\epsilon}{\sqrt 2}}^{\tau_+ +\frac{\epsilon}{\sqrt 2}} \int_{|z^{\prime}| <\delta}\left( \beta+\frac{\kappa}{b}\right) \mathrm{e}^{-4\rho \Im \varphi} \chi_{\delta}^3 a_{0}a_{0}\bar a_0^{(0)} |\partial_t \varphi|^2 \partial_t \varphi b^{\frac{3}{2}}d\tau \wedge dz^{\prime}+\mathcal{O}(\rho^{-1}).
\end{align*}
By combining an argument analogous to the stationary phase method with the injectivity of the weighted ray transform,
 we derive that
\begin{equation}\label{relation}
\beta+\frac{\kappa}{b}=0. 
\end{equation}
Considering the terms of order $\rho^{\frac{1}{2}}$, we have 
\begin{align}
& \rho^{-\frac{1}{2}}\int_0^T\int_M \partial_t u_0 \left(\beta \partial_tu_1\partial_tu_2+\kappa \langle \nabla_g u_1,\nabla_g u_2 \rangle_g \right) dV_g  dt \nonumber \\
=& \rho^{-\frac{1}{2}}\int_0^T\int_M \mathrm{e}^{-2\mathrm{i}\rho \bar \varphi}(-2\mathrm{i} \rho \partial_t \bar \varphi \bar a_0^{(0)}+\partial_t \bar a_0^{(0)}-\mathrm{i}\partial_t \bar \varphi \bar a_1^{(0)})(\beta \mathrm{e}^{2\mathrm{i} \rho \varphi}(\mathrm{i} \rho \partial_t \varphi a_{0}+\partial_t a_{0}+\mathrm{i}\partial_t \varphi a_{1})^2 \nonumber \\
&+\kappa\mathrm{e}^{2\mathrm{i}\rho \varphi}\langle \mathrm{i} \rho a_{0} d\varphi+da_{0}+\mathrm{i}a_{1} d\varphi, \mathrm{i} \rho a_{0} d\varphi+da_{0}+\mathrm{i}a_{1} d\varphi  \rangle_{g} )dV_g dt \nonumber \\
=&\rho^{-\frac{1}{2}}\int_0^T\int_M \mathrm{e}^{-4 \rho \Im \varphi}(-2\mathrm{i} \rho \partial_t \bar \varphi \bar a_0^{(0)}+\partial_t \bar a_{0}^{(0)}-\mathrm{i}\partial_t \bar \varphi \bar a_1^{(0)})(-\rho^2(\beta+\frac{\kappa}{b})(\partial_t\varphi)^2 a_{0} a_{0} \nonumber \\
&+2\beta \mathrm{i} \rho \partial_t \varphi a_{0}(\partial_t a_{0}+ \mathrm{i}\partial_t \varphi a_{1})+2\kappa  \mathrm{i} \rho(\langle d \varphi ,da_{0}  \rangle_ga_{0}+ \mathrm{i} \langle d\varphi ,d\varphi  \rangle_ga_{0}a_{1}+\mathcal{O}(1)))dV_g  dt \nonumber \\
=& \rho^{-\frac{1}{2}}\int_0^T\int_M \mathrm{e}^{-4 \rho \Im \varphi}(-2\mathrm{i} \rho) \partial_t \bar \varphi \bar a_{0}^{(0)}  (2\mathrm{i}\rho)(a_{0}(\beta \partial_t \varphi \partial_t a_{0}-b \beta \langle d\varphi ,d a_{0} \rangle_g) \nonumber \\
&+\mathrm{i}(\beta+\frac{\kappa}{b})(\partial_t \varphi)^2 a_{0} a_{1})dV_g dt+\mathcal{O}(\rho^{-1}) \nonumber \\
=&4 \rho^{\frac{3}{2}} \int_0^T\int_M \beta \mathrm{e}^{-4 \rho \Im \varphi} \partial_t \bar \varphi \bar a_{0}^{(0)} a_{0}(\partial_t \varphi \partial_t a_{0}-b  \langle d\varphi ,d a_{0} \rangle_g)dV_g dt +\mathcal{O}(\rho^{-1}) \nonumber \\
=& -4 \rho^{\frac{3}{2}} \int_0^T\int_M \beta \mathrm{e}^{-4 \rho \Im \varphi} \partial_t \bar \varphi \bar a_{0}^{(0)} a_{0}  \langle d\varphi ,d a_{0} \rangle_{\bar g} dV_g dt +\mathcal{O}(\rho^{-1}) \nonumber \\
=& -4 \rho^{\frac{3}{2}} \int_{\tau_{-}-\frac{\epsilon}{\sqrt 2}}^{\tau_+ + \frac{\epsilon}{\sqrt 2}} \int_{|z^{\prime}|<\delta} \beta \mathrm{e}^{-4 \rho \Im \varphi} \partial_t \bar \varphi \bar a_{0}^{(0)} a_{0}  \langle d\varphi ,d a_{0} \rangle_{\bar g} b^{\frac{3}{2}} d \tau \wedge d z^{\prime} +\mathcal{O}(\rho^{-1}), \label{simp}
\end{align}
with sufficiently small $\delta$.
Notice that in the null geodesic $\varsigma$
\begin{equation*}
 \langle d\varphi ,d a_{0} \rangle_{\bar g}=\partial_{\tau} a_{0}.
\end{equation*}
Using the method of stationary phase, we simplify the above equation \eqref{simp} to 
\begin{equation*}
C\int_{\tau_{-}-\frac{\epsilon}{\sqrt 2}}^{\tau_+ + \frac{\epsilon}{\sqrt 2}} \beta \partial_{\tau} a_{0} d\tau+\mathcal{O}(\rho^{-1}).
\end{equation*}
Using the fact that $\beta=0$ on $\partial M$ from the assumption $\beta_j=0$ on $\partial M$, integrating by parts and 
letting $\rho \to +\infty$ yield that 
\begin{equation}
\int_{\sigma} (\partial_{\tau} \beta) a_{0} d\tau=\int_{\sigma} \partial_{\tau} \beta b^{-\frac{1}{4}}\mathrm{e}^{-\frac{\sqrt 2}{4}\int_{\tau_{-}}^{\tau}\gamma(s,0,0,0)ds}(\det Y(\tau))^{-\frac{1}{2}} d\tau=0.
\end{equation}
Using the injectivity of the weighted ray transform under the Geometric Assumption, we obtain
\[
\partial_\tau \beta = 0 \quad \text{along } \sigma.
\]
Since $\beta=0$ on $\partial M$ and every geodesic hits the boundary in a finite time, we obtain that 
\begin{equation*}
\beta=0   \quad \text{in } M.
\end{equation*}
Using equation \eqref{relation}, we have $\kappa=0$ in $M$. Then we complete the proof.
\end{proof}

\section{Conclusion and future work} \label{conclusion}
In this paper, we established uniqueness in inverse boundary value problems for the Jordan--Moore--Gibson--Thompson equation with quadratic nonlinearities of Westervelt and Kuznetsov type. The associated Dirichlet-to-Neumann map uniquely determines $\beta$ in the Westervelt-type model and $(\beta,\kappa)$ in the Kuznetsov-type model. The results hold in bounded Euclidean domains and on compact Riemannian manifolds under suitable geometric assumptions. The proof combines second order linearization with geometric optics and Gaussian beam constructions for the linearized MGT equation, reducing the inverse problem to the injectivity of weighted ray transforms.

In Theorem \ref{thm1}, we consider the case $\gamma> 0$. We note that our method remains applicable in the case $\gamma<0$.  The critical regime $\gamma=0$ corresponds to conservative-type dynamics and the absence of dissipation, and it would be of interest to investigate whether our uniqueness results can be extended to this degenerate case.

Beyond these questions, there are several further directions to explore. For the linear MGT equation, it is natural to ask whether the Dirichlet-to-Neumann map uniquely determines the coefficients $(\alpha,c)$ on a compact Riemannian manifold. For the nonlinear setting, quantitative stability estimates for recovering nonlinear coefficients from the DtN map on Riemannian manifolds are also of clear interest. Finally, it would be important to develop effective reconstruction algorithms for the JMGT equation. Our approach provides a theoretical foundation for reconstructing the nonlinear coefficients from boundary measurements.

\section*{Acknowledgements}
 X. Xu is partially supported by the National Key Research and Development Program of China (No. 2024YFA1012303), National Natural Science Foundation of China (No. 12525112), and the Open Research Project of Innovation Center of Yangtze River Delta, Zhejiang University. T. Zhou is partially supported by the National Key Research and Development Program of China (No. 2024YFA1012301), the Zhejiang Provincial Basic Public Welfare Research Program [
Grant Number LDQ24A010001], and NSFC Grant 12371426.

\bibliographystyle{amsplain}
\bibliography{ref}
\end{document}